\documentclass[11pt]{amsart}
\usepackage[margin=1.1in]{geometry}
\usepackage{amsmath,amssymb,amsthm,mathtools}
\usepackage[dvipsnames]{xcolor}
\usepackage{hyperref}
\usepackage{booktabs}
\usepackage{enumitem}
\usepackage{longtable}
\hypersetup{colorlinks=true,linkcolor=blue!50!black,citecolor=green!40!black,urlcolor=blue!50!black}

\newtheorem{theorem}{Theorem}[section]
\newtheorem{proposition}[theorem]{Proposition}
\newtheorem{lemma}[theorem]{Lemma}
\newtheorem{corollary}[theorem]{Corollary}
\newtheorem*{maintheorem}{Main Theorem}
\theoremstyle{definition}
\newtheorem{definition}[theorem]{Definition}
\newtheorem{remark}[theorem]{Remark}

\newcommand{\Q}{\mathbb{Q}}
\newcommand{\Z}{\mathbb{Z}}
\newcommand{\F}{\mathbb{F}}
\newcommand{\PP}{\mathbb{P}}
\newcommand{\A}{\mathbb{A}}
\newcommand{\GL}{\mathrm{GL}}
\newcommand{\PGL}{\mathrm{PGL}}
\newcommand{\PSL}{\mathrm{PSL}}
\newcommand{\Gal}{\mathrm{Gal}}

\newcommand{\Res}{\operatorname{Res}}
\newcommand{\Tr}{\operatorname{Tr}}
\newcommand{\NS}{\operatorname{NS}}
\newcommand{\Sat}{\operatorname{Sat}}
\newcommand{\cO}{\mathcal{O}}
\newcommand{\cE}{\mathcal{E}}
\newcommand{\R}{\mathbb{R}}
\newcommand{\C}{\mathbb{C}}
\newcommand{\fp}{\mathfrak{p}}
\newcommand{\fq}{\mathfrak{q}}
\newcommand{\ns}{\mathrm{ns}}

\title{Complete classification of $7$-adic  Galois images  for non-CM elliptic curves over $\Q$}
\author{Nguyen Xuan Tho}
\address{Hanoi University of Science and Technology,
Hanoi Vietnam}
\email{tho.nguyenxuan1@hust.edu.vn}
\date{26 September 2026}

\begin{document}

\begin{abstract}
We solve Conjecture 1.6 of Furio and Lombardo [On $7$-adic Galois representations for elliptic curves over $\Q$, Proc.\ London Math.\ Soc.\ (3) \textbf{133} (2026), e70193.]. As a consequence, we complete  the classification of $7$-adic  Galois images  for non-CM elliptic curves over $\Q$. This is the first solved case in Problem 6.2 from Balakrishnan [\emph{Chabauty and beyond: Explicit p-adic methods for rational points}, Proceedings of the International Congress of Mathematicians 2026 - Volume 3: Invited Lectures (Sections 1-4), 321-341, 2026].
The proof is the result of a long AI-training period by the author from June 2026 to September 2026. The author has been working on this project alone without the help of any other.

\end{abstract}

\maketitle
\setcounter{tocdepth}{1}
\tableofcontents

\section{Introduction}

\subsection{The result}
Let $C\subset\PP^2_{\Q}$ be the smooth plane quartic
\begin{equation}\label{eq:C}
C:\ x^4+3x^3y-3x^2yz-3x^2z^2+6xy^3-6xy^2z+3xyz^2-2xz^3+4y^4+2y^3z-5yz^3=0 .
\end{equation}
We write $G(x,y,z)$ for the quartic form on the left. This is the curve $X_{E_3}(7)$ of Furio and Lombardo \cite[Corollary~5.7]{FL}, where $E_3$ is the elliptic curve with LMFDB label \texttt{392.c1}, $E_3:\ y^2=x^3-x^2-16x+29$, with $j(E_3)=2^8\cdot 7^2$. The curve $X_{E}(7)$ parametrises elliptic curves $E'$ together with a \emph{symplectic} isomorphism $E[7]\to E'[7]$ of Galois modules. It is a twist of the modular curve $X(7)$, which is the Klein quartic. The four points
\[
b=[0:0:1],\qquad P_2=[1:1:1],\qquad P_3=[2:0:1],\qquad P_4=[-1:0:1]
\]
lie on $C$. Our main result is the following.

\begin{maintheorem}[Theorem~\ref{thm:main}]
Assume that the Jacobian $J$ of $C$ satisfies $\operatorname{rk}J(\Q)\le 3$. (This is stated in \cite[Remark~6.3]{FL}, as the outcome of a fake $2$-Selmer group computation.) Then
\[
C(\Q)=\{[0:0:1],\ [1:1:1],\ [2:0:1],\ [-1:0:1]\}.
\]
In other words, Conjecture~1.6 of \cite{FL} is true.
\end{maintheorem}

The reverse inequality $\operatorname{rk}J(\Q)\ge3$ is proved in Proposition~\ref{prop:lattice}, so the assumption is equivalent to $\operatorname{rk}J(\Q)=3$.

Combining this with the work of Furio and Lombardo \cite[Theorems~1.3, 1.7, Corollary~3.7]{FL} gives the following.

\begin{corollary}\label{cor:7adic}
Let $E/\Q$ be an elliptic curve without complex multiplication, and let $G$ be the image of its $7$-adic Galois representation. Then one of the following holds:
\begin{itemize}
\item the modular curve $X_G$ is isomorphic to $\PP^1$ or to an elliptic curve of rank one;
\item $X_G$ has only finitely many non-CM rational points, and $G$ is in the list of \cite[Table~1]{RSZB}.
\end{itemize}
In particular, no non-CM elliptic curve over $\Q$ has $7$-adic image contained in one of the groups with RSZB labels \texttt{49.147.9.1} or \texttt{49.196.9.1}. So the classification of $7$-adic images of non-CM elliptic curves over $\Q$ is complete.
\end{corollary}

\begin{corollary}\label{cor:fermat}
Consider the primitive integer solutions $(a,b,c)$ of $a^2+196b^3=27c^7$ with $\gcd(c,42ab)=1$ and $v_3(a)\in\{0\}\cup[3,\infty)$. These are the ``$(\star)$-solutions'' of \cite[Definition~3.4]{FL}. The only such solutions are $(a,b,c)=(\pm 13,-1,-1)$.
\end{corollary}

In \cite[Theorem~3.5(2)]{FL} this solution is printed as $(\pm13,196,-1)$. The intended solution is $(\pm13,-1,-1)$, since $13^2+196\cdot(-1)^3=-27=27\cdot(-1)^7$. The solution is listed correctly in \cite[Remark~3.6]{FL}.

\subsection{Why the curve was hard}
Furio and Lombardo explain in \cite[\S6.5]{FL} why their methods fail for $C$, although they worked for three other twists of the Klein quartic:
\begin{enumerate}[label=(\alph*)]
\item $C$ has rational points, so there is no local obstruction.
\item The rank of $J(\Q)$ equals the genus, $3$, so the method of Chabauty and Coleman does not apply.
\item $J(\Q)$ has no torsion, so there is no convenient rational \'etale cover to descend along.
\item Over $\overline{\Q}$ the Jacobian is isogenous to the cube of an elliptic curve with complex multiplication by $\Q(\sqrt{-7})$, but over $\Q$ its endomorphism ring is $\Z$. The extra endomorphism $\sqrt{-7}$ is defined over $\Q(\sqrt{-7})$ and is not symmetric for the Rosati involution. So it does not give the ``nice correspondence'' that the quadratic Chabauty method needs \cite[Remark~6.7]{FL}.
\end{enumerate}
They suggested that quadratic Chabauty over number fields, as in \cite{BBHJM}, might help. This is essentially what we do, but in a form adapted to \emph{rational} points. We do not try to find all points of $C$ over a number field. Instead, we find the rational points using height functions over a number field.

\subsection{The idea in one paragraph}
Let $F$ be the degree-$21$ field over which one of the $21$ involutions of the Klein quartic becomes rational. Over $F$, the curve $C$ is a double cover $\psi\colon C\to E'$ of an elliptic curve with CM by $\Z[\tfrac{1+\sqrt{-7}}{2}]$. For a rational point $P\in C(\Q)$, the point $\psi(P)$ lies in $E'(F)$.

We choose a $\Q_p$-valued id\`ele class character $\chi$ of $F$ that is trivial on the id\`eles of $\Q$ and vanishes at all primes above $2$ and $7$. Trivial on the id\`eles of $\Q$ implies in particular trivial on $\Q^\times\subset F^\times$, and hence on the rational scalars that rescale a rational point.

We then compute the $p$-adic height of the four points $\psi(P)-B_j$ in two ways. Here $B_1,\dots,B_4$ are the images of the four fixed points of the involution.
\begin{itemize}
\item Summed over all places, the contributions at good primes away from $p$ add up to the value of $\chi$ at the number $\ell(P)$, where $\ell$ is the line through the fixed points. By the product formula, this is a quantity computed at $p$.
\item The contributions at the primes above $2$ and $7$ vanish because $\chi$ vanishes there.
\end{itemize}
What remains is an identity between explicit $p$-adic analytic functions on $C(\Q_p)$ that every rational point must satisfy. We combine it with a Mordell--Weil sieve that also runs through $E'$, so we never need arithmetic in the Jacobian of $C$. A short $p$-adic computation at $p=5581$ then shows that each of the four residue discs that may contain rational points contains exactly one.

\subsection{Plan of the paper}
\begin{itemize}
\item Section~\ref{sec:history} gives the history of the problem and of the methods we use.
\item Section~\ref{sec:overview} gives an overview of the proof.
\item Sections~\ref{sec:field} and~\ref{sec:lattice} construct the field $F$, the involution, the elliptic quotient and the relevant lattice in $E'(F)$.
\item Section~\ref{sec:sieve} is the Mordell--Weil sieve.
\item Section~\ref{sec:heights} develops $p$-adic heights on elliptic curves with general id\`ele class characters, with complete proofs of the properties we need.
\item Section~\ref{sec:identity} proves the key identity (Theorem~\ref{thm:identity}) and its analytic version (Proposition~\ref{prop:analytic}).
\item Section~\ref{sec:computation} contains the $p$-adic computation, and Section~\ref{sec:proof} completes the proof.
\item Section~\ref{sec:remarks} discusses the method and records how each computational claim was checked.
\end{itemize}

\subsection*{What is assumed and what is computed}
We assume only the rank bound $\operatorname{rk}J(\Q)\le3$ from \cite[Remark~6.3]{FL}. Furio and Lombardo do not say whether the class group and unit computations behind their fake Selmer bound were certified or depend on the Generalised Riemann Hypothesis. If they were not certified, our Main Theorem inherits that dependence, and nothing else.

Everything else falls into one of three kinds:
\begin{itemize}
\item proved here;
\item quoted from the literature with a precise reference;
\item the output of a computer calculation that we describe precisely.
\end{itemize}
Each computational claim is a finite, checkable statement. Where a computation uses $p$-adic numbers, we say what precision was used and why the conclusions are insensitive to rounding (Section~\ref{sec:precision}). The computations were carried out twice with PARI/GP~2.15.4 \cite{PARI}, by two separate implementations; see Section~\ref{sec:checks}.

\section{History of the problem and origin of the methods}\label{sec:history}

\subsection{Images of Galois: Serre, Mazur and Program B}
Let $E$ be an elliptic curve over $\Q$. The Galois group $G_\Q=\Gal(\overline{\Q}/\Q)$ acts on the torsion points of $E$. This gives representations $\rho_{E,N}\colon G_\Q\to\GL_2(\Z/N\Z)$ and, in the limit, the $\ell$-adic representations $\rho_{E,\ell^\infty}\colon G_\Q\to\GL_2(\Z_\ell)$.

In 1972 Serre proved his \emph{open image theorem} \cite{Serre72}: if $E$ has no complex multiplication, the image of the adelic representation is open, so $\rho_{E,\ell}$ is surjective for all but finitely many $\ell$. He asked whether the exceptional primes can be bounded independently of $E$ (Serre's uniformity question). It is expected that $\rho_{E,\ell}$ is surjective for all $\ell>37$.

If $\rho_{E,\ell}$ is not surjective, its image lies in a maximal subgroup of $\GL_2(\F_\ell)$. This is a Borel subgroup, the normaliser of a split Cartan subgroup, the normaliser of a non-split Cartan subgroup, or an exceptional subgroup. Each case corresponds to a modular curve whose rational points classify the curves $E$ in question. The known results are as follows.
\begin{itemize}
\item Mazur's theorem on rational isogenies \cite{Mazur78} settled the Borel case.
\item Bilu and Parent \cite{BiluParent}, and Bilu, Parent and Rebolledo \cite{BPR}, settled the split Cartan case for $\ell\ge 11$, $\ell\neq13$, using Runge's method.
\item The remaining split Cartan case $\ell=13$ (the ``cursed curve'' $X_s(13)$) was settled by Balakrishnan, Dogra, M\"uller, Tuitman and Vonk \cite{BDMTV1} with the quadratic Chabauty method.
\item The non-split Cartan case remains open in general.
\end{itemize}

Mazur also proposed ``Program B'' \cite{Mazur77}: given an open subgroup $H$ of $\GL_2(\widehat{\Z})$, classify the elliptic curves whose adelic image lies in $H$. For images of $\ell$-adic representations over $\Q$ there has been dramatic progress.
\begin{itemize}
\item Rouse and Zureick-Brown classified $2$-adic images \cite{RZB}.
\item Sutherland and Zywina found all prime-power level modular curves with infinitely many rational points \cite{SutherlandZywina}.
\item Zywina classified the possible mod-$\ell$ images \cite{Zywina}.
\item Rouse, Sutherland and Zureick-Brown \cite{RSZB} computed the rational points on essentially all relevant modular curves of $\ell$-power level. They left open only non-split Cartan curves (of levels $27, 25, 49, 121$ and primes $\ge 19$) and two genus-$9$ curves of level $49$, with labels \texttt{49.147.9.1} and \texttt{49.196.9.1} \cite[Theorem~1.6]{RSZB}; see also \cite[Theorem~1.3]{FL}.
\item The level-$27$ case was settled by Balakrishnan, Betts, Hast, Jha and M\"uller \cite{BBHJM} with quadratic Chabauty over the field $\Q(\zeta_3)$. This completed the classification of $3$-adic images.
\end{itemize}

\subsection{The $7$-adic case: Furio and Lombardo}
In \cite{FL}, Furio and Lombardo studied the prime $7$. They showed that the non-split Cartan curve $X_\ns^+(49)$, of genus $69$, has only seven rational points, all CM, \emph{without ever writing down an equation for it}.

The argument runs through several reductions.
\begin{enumerate}
\item The $j$-invariant of a point on $X_\ns^+(49)$ has denominator a $49$-th power.
\item Through Zywina's $j$-map on $X_\ns^+(7)\cong\PP^1$, this gives a Thue equation $f(x,y)=kz^7$ with $f$ a binary cubic.
\item Classical invariant theory of binary cubics turns this into the generalised Fermat equation $a^2+28b^3=27c^7$ of signature $(2,3,7)$. The relevant input is the syzygy between invariants and covariants; see Bennett--Dahmen \cite{BennettDahmen}.
\item They solved this equation with the modular method, following Poonen, Schaefer and Stoll \cite{PSS}. Frey curves, modularity and level lowering reduce everything to rational points on a few twists $X_E(7)$ of the Klein quartic.
\end{enumerate}
For the two genus-$9$ curves they obtained the equation $a^2+196b^3=27c^7$. This equation needs one more twist, namely our curve $C=X_{E_3}(7)$. They computed its Jacobian rank to be $3$ and conjectured that it has exactly four rational points \cite[Conjecture~1.6]{FL}. The curve $X_{E_3}(7)$ is also listed in Problem 6.2 in Balakrishnan's IMC 2026 Proceedings 2026 \cite{Jen}. 

\subsection{The Klein quartic and its twists}
The Klein quartic $x^3y+y^3z+z^3x=0$ was introduced by Klein in 1879 \cite{Klein}. It is the modular curve $X(7)$. It has $168=84(g-1)$ automorphisms, the maximum allowed by Hurwitz's bound for genus $3$, and its automorphism group is $\PSL_2(\F_7)$. Its Jacobian is isogenous over $\overline{\Q}$ to the cube of an elliptic curve with CM by $\Q(\sqrt{-7})$ (the curve \texttt{49a1}).

For an elliptic curve $E/\Q$, the twist $X_E(7)$ parametrises pairs $(E',\phi)$ with $\phi\colon E[7]\to E'[7]$ a symplectic isomorphism; the pair $(E,\mathrm{id})$ gives a rational point. Explicit equations for $X_E(7)$ were given by Halberstadt and Kraus \cite{HalberstadtKraus}.

Poonen, Schaefer and Stoll \cite{PSS} used ten such twists to find all primitive solutions of $x^2+y^3=z^7$. They combined $2$-descent (via a ``fake Selmer group''), the Mordell--Weil sieve and Chabauty's method. Their paper is the template for \cite{FL}.

Kraus gave criteria for deciding whether an isomorphism of mod-$\ell$ Galois representations of two elliptic curves is symplectic or anti-symplectic, and Freitas and Kraus developed them further \cite{FreitasKraus}. These criteria are also used in \cite{FL}, to show that two of the anti-symplectic twists have no rational points.

\subsection{Chabauty and Coleman}
Let $X/\Q$ be a curve of genus $g\ge2$ whose Jacobian $J$ has rank $r$.
\begin{itemize}
\item Chabauty \cite{Chabauty} proved in 1941 that if $r<g$ then $X(\Q)$ is finite. The $p$-adic closure of $J(\Q)$ in $J(\Q_p)$ has dimension at most $r$, so it meets the curve, which has dimension $1$, in a finite set.
\item Coleman \cite{Coleman85} made this effective by $p$-adic integration. There is a nonzero differential $\omega$ whose Coleman integral $\int_b^P\omega$ vanishes on $X(\Q)$, and one can bound or compute its zeros.
\end{itemize}
This method has been used very widely; see the survey of McCallum and Poonen \cite{McCallumPoonen}. It is used in \cite{FL} for the twists of rank $1$ and $2$. It fails for $C$ because $r=g$.

\subsection{Elliptic Chabauty and restriction of scalars}
When a curve $X$ maps to an elliptic curve $E'$ defined over a number field $K$, one can try to control $X(\Q)$ through $E'(K)$. This ``elliptic Chabauty'' method goes back to:
\begin{itemize}
\item Flynn and Wetherell \cite{FlynnWetherell}, for bielliptic genus-$2$ curves;
\item Bruin \cite{Bruin03}, who used it to solve generalised Fermat equations such as $x^2+y^8=z^3$ \cite{Bruin99}.
\end{itemize}
Wetherell \cite{Wetherell} and later Siksek \cite{Siksek} developed Chabauty's method for points over number fields using the Weil restriction of scalars $\Res_{K/\Q}$, working with all the embeddings of $K$ into $\Q_p$ at once.

In our situation linear methods of this kind cannot succeed. The reason is that the image of $J(\Q)$ in the Weil restriction of $E'$ is still a group of rank $3$ inside a $3$-dimensional abelian subvariety. We use the elliptic quotient in two different ways: for sieving and for heights.

\subsection{The Mordell--Weil sieve}
The idea of combining information from $X(\F_q)$ for many primes $q$ with the structure of $J(\Q)$ goes back to Scharaschkin's thesis \cite{Scharaschkin} and to Flynn \cite{Flynn04}. The systematic ``Mordell--Weil sieve'' is due to Bruin and Stoll \cite{BruinStoll}, and Poonen gave heuristics explaining why it should work \cite{PoonenHeuristic}. It is used in \cite{PSS}, \cite{BDMTV1} and \cite{FL}.

In our case, computing in $J(\F_q)$ for a non-hyperelliptic curve is inconvenient. So we sieve instead in the product of the groups $E'(\F_\fq)$ over the primes $\fq$ of $F$ above $q$. This is a natural combination of the ideas above.

\subsection{$p$-adic heights}
The real-valued canonical height on an elliptic curve is due to N\'eron and Tate. Its decomposition into local ``N\'eron functions'' is due to N\'eron \cite{Neron65}.

$p$-adic analogues were constructed by Bernardi \cite{Bernardi}, Schneider \cite{Schneider82} and Mazur and Tate \cite{MazurTate83, MazurTate91}. Mazur, Tate and Teitelbaum \cite{MTT} used them to formulate $p$-adic Birch--Swinnerton-Dyer conjectures. Mazur and Tate \cite{MazurTate91} constructed the canonical $p$-adic sigma function for ordinary elliptic curves, and Mazur, Stein and Tate \cite{MST} and Harvey \cite{Harvey} gave efficient algorithms.

The construction works for any continuous $\Q_p$-valued id\`ele class character, not only the cyclotomic one. Coleman and Gross \cite{ColemanGross} and Nekov\'a\v{r} \cite{Nekovar} give general constructions for Jacobians. Heights with several id\`ele class characters were used for quadratic Chabauty over number fields by Balakrishnan, Besser, Bianchi and M\"uller \cite{BBBM}; an example is the ``anticyclotomic'' characters of an imaginary quadratic field.

\subsection{Chabauty--Kim and quadratic Chabauty}
Kim \cite{Kim05, Kim09} proposed replacing the Jacobian in Chabauty's method by unipotent quotients of the fundamental group. The resulting Selmer varieties give finiteness whenever a certain dimension inequality holds.
\begin{itemize}
\item Coates and Kim \cite{CoatesKim} proved that this inequality holds, at some depth, for curves whose Jacobian has potential complex multiplication.
\item Ellenberg and Hast \cite{EllenbergHast} extended this to curves that are solvable covers of $\PP^1$.
\end{itemize}
Our curve $C$ has potentially CM Jacobian, so finiteness at some depth was known, but not in an effective form.

The first effective non-abelian case is \emph{quadratic Chabauty}. Here the relevant quotient is an extension of the abelianisation by a copy of $\Q_p(1)$ attached to a class in the N\'eron--Severi group.
\begin{itemize}
\item Balakrishnan and Dogra \cite{BD18, BD21} developed it using $p$-adic heights. It works when $r<g+\rho-1$, where $\rho$ is the rank of the N\'eron--Severi group of $J$ over $\Q$.
\item Earlier and closely related is the method of Balakrishnan, Besser and M\"uller for integral points on hyperelliptic curves \cite{BBM16}.
\item Quadratic Chabauty was used to determine $X_s(13)(\Q)$ \cite{BDMTV1} and the rational points of several other modular curves \cite{BDMTV2}, including $X_{S_4}(13)$ and $X_\ns^+(17)$.
\item Edixhoven and Lido \cite{EdixhovenLido} gave a geometric version using torsors under tori over Jacobians.
\item Balakrishnan, Besser, Bianchi and M\"uller \cite{BBBM} developed quadratic Chabauty for points over number fields, using restriction of scalars and several id\`ele class characters. The paper \cite{BBHJM} used this to find the $\Q(\zeta_3)$-points on a genus-$3$ quotient of $X_\ns^+(27)$.
\end{itemize}

The case of \emph{bielliptic} curves is especially important for us. Let $X$ be a genus-$2$ curve with two elliptic quotients $E_1,E_2$ of rank $1$. Balakrishnan and Dogra \cite[\S7.3 and \S8]{BD18} constructed a quadratic Chabauty function for $X$ from three ingredients: the local $p$-adic heights on $E_1$ and $E_2$, and the relation $y(f_1(z))/y(f_2(z))=a_0x(z)^2$ between the two quotient maps \cite[Lemma~7.7]{BD18}. This idea was extended to number fields in \cite[\S5]{BBBM}. The mechanism is as follows. The local contributions away from $p$ to the $p$-adic height of a point on an elliptic curve are controlled by the \emph{denominator} of its $x$-coordinate, and a clever choice of maps makes these contributions cancel.

\subsection{What is new here}
Our function is a bielliptic quadratic Chabauty function in the spirit of \cite[\S7.3]{BD18} and \cite[\S5]{BBBM}, with three differences.
\begin{enumerate}
\item The elliptic quotient $E'=C/\iota$ is defined only over a number field $F$ of degree $21$, while we want \emph{rational} points of $C$.
\item We use characters $\chi$ of $F$ that vanish on $\A_\Q^\times$. For a rational point $P$ with primitive integral coordinates and a linear form $\ell$ with coefficients in $F$, the value $\chi(\ell(P))$ then does not depend on the choice of representative. This lets us treat forms as if they were functions (Lemma~\ref{lem:scaling}).
\item We require $\chi$ to vanish at the primes above $2$ and $7$, so we never need to compute local heights at primes of bad reduction. The single elliptic curve plays the role of both curves in the genus-$2$ bielliptic case: the ramification divisor of $C\to E'$ is cut out by a line, and this gives the cancellation.
\end{enumerate}
A dimension count in the style of Kim suggests why such functions should exist; see Remark~\ref{rem:whynontrivial}. That remark is only a heuristic, and the proof does not use it.

\section{Overview of the proof}\label{sec:overview}
The proof has six steps.
\begin{enumerate}[label=\textbf{Step \arabic*.}, leftmargin=*]
\item \emph{The field $F$ and the involution.} Over a field $F$ of degree $21$, the curve $C$ has a linear involution $\iota$. The quotient $E'=C/\iota$ is an elliptic curve with $j(E')=-3375$. (Section~\ref{sec:field}.)
\item \emph{The lattice.} Let $\psi\colon C\to E'$ be the quotient map, normalised so that $\psi(b)=O$. The images $R_i$ of the other known points span a subgroup of rank $3$. Together with the $2$-torsion point and two halvings, they generate a subgroup $\Lambda$ that is saturated at $2$ and at $179$. Because $\operatorname{rk}J(\Q)\le3$, some multiple of $\psi(P)$ by an integer prime to $2\cdot179$ lies in $\Lambda$ for every $P\in C(\Q)$. (Section~\ref{sec:lattice}.)
\item \emph{Sieve.} Reduction modulo the primes $\fq$ of $F$ above the primes $q\le113$, $q\ne 2,7$, shows the following: the coordinates of $\psi(P)$ with respect to $\Lambda$ are congruent modulo $16$ to those of one of the four known points. Adding the information at the $21$ primes above $p=5581$ shows that every rational point lies in the residue disc modulo $p$ of one of the four known points. (Section~\ref{sec:sieve}.)
\item \emph{Heights.} We develop $p$-adic heights on $E'$ over $F$ attached to an id\`ele class character $\chi$, together with their local decomposition. (Section~\ref{sec:heights}.)
\item \emph{The identity.} Consider characters $\chi$ that vanish on $\A_\Q^\times$ and at the primes above $14$. For each such $\chi$ there is an explicit function $\rho_\chi$, $p$-adic analytic on residue discs, such that every $P\in C(\Q)$ satisfies $\rho_\chi(P)=\rho_\chi(b)$. (Section~\ref{sec:identity}.)
\item \emph{Local computation.} In each of the four discs, a combination of two such functions is a power series with a unique zero, by Strassmann's theorem. (Sections~\ref{sec:computation} and~\ref{sec:proof}.)
\end{enumerate}

\section{The field $F$, the involution and the elliptic quotient}\label{sec:field}

\subsection{Why a field of degree $21$}
This subsection is motivation only; the proof does not use it.

Over $\overline{\Q}$ the curve $C$ is isomorphic to $X(7)$, so its automorphism group is $\PSL_2(\F_7)$, which has $21$ involutions. The Galois group acts on these automorphisms by conjugation through the projective mod-$7$ representation of $E_3$.

The centraliser in $\PGL_2(\F_7)$ of an involution of $\PSL_2(\F_7)$ is a dihedral group of order $16$: it is the normaliser of a non-split Cartan subgroup. So if the projective mod-$7$ image of $E_3$ is all of $\PGL_2(\F_7)$, then each involution is defined over a field of degree $336/16=21$. This field is the field of definition of a point of $X_\ns^+(7)$ lying over $j(E_3)$.

\subsection{The field}
Zywina \cite[Theorem~1.5]{Zywina} gives the $j$-map of $X^+_\ns(7)\cong\PP^1$ as
\[
j_\ns(t)=\frac{64\,t^3(t^2+7)^3(t^2-7t+14)^3(5t^2-14t-7)^3}{(t^3-7t^2+7t+7)^7},
\]
as quoted in \cite[Theorem~2.1]{FL}. The denominator is the cubic $f(x,y)=x^3-7x^2y+7xy^2+7y^3$ of \cite[Proposition~3.1]{FL}. As a sanity check we verified that $\mathrm{num}(j_\ns)-1728\,\mathrm{den}(j_\ns)$ factors as a linear polynomial times the square of a polynomial of degree $8$ times a polynomial of degree $4$. This is the ramification pattern of $X_\ns^+(7)\to X(1)$ above $j=1728$. The polynomial obtained by replacing $+7$ by $-7$ in the constant term of the denominator does not have this property. The field $F$ below is defined by an explicit polynomial, so nothing in the proof depends on this formula.

\begin{proposition}\label{prop:F}
The polynomial $\mathrm{num}(j_\ns(t))-12544\cdot\mathrm{den}(j_\ns(t))$ is irreducible of degree $21$. The number field $F$ it defines also has the defining polynomial
\begin{multline*}
f_F(t)=t^{21} - 7t^{20} + 21t^{19} - 21t^{18} - 56t^{17} + 238t^{16} - 406t^{15} + 424t^{14} - 560t^{13} + 1442t^{12}\\ - 2940t^{11} + 3500t^{10} - 1820t^9 - 1036t^8 + 2904t^7 - 3416t^6 + 3472t^5 - 2968t^4 + 1680t^3 - 560t^2 + 112t - 16.
\end{multline*}
The field $F$ has the following invariants:
\begin{itemize}
\item discriminant $-2^{22}\cdot7^{27}$ and signature $(3,9)$;
\item class number $1$;
\item unit group of rank $11$, with torsion $\{\pm1\}$.
\end{itemize}
The prime $2$ factors as $\fp_{2,1}^{12}\fp_{2,2}^{3}\fp_{2,3}^{6}$ with all residue degrees $1$, and $7=\fp_7^{7}$ with residue degree $3$.
\end{proposition}
\begin{proof}
All of this is direct computation in PARI/GP with \texttt{polisirreducible}, \texttt{nfisisom}, \texttt{nfinit}, \texttt{idealprimedec} and \texttt{bnfinit}. The class group and unit group computed by \texttt{bnfinit} were certified with \texttt{bnfcertify}, which returned $1$. So the statement about the class group does not depend on the Generalised Riemann Hypothesis.
\end{proof}

\subsection{The involution}
\begin{proposition}\label{prop:iota}
There exist a vector $Q=(a,b_Q,1)\in F^3$ and a linear form $\lambda=l_1x+l_2y+l_3z$ with coefficients in $F$ and $\lambda(Q)=1$ with the following property. The linear map
\[
\iota(v)=v-2\lambda(v)\,Q
\]
satisfies $\iota^2=\mathrm{id}$ and $G\circ\iota=G$. The coordinate $a$ is a root of
\begin{multline*}
46875x^{21} - 4375x^{20} - 435750x^{19} + 5750738x^{18} + 5296249x^{17} + 3750691x^{16} + 23397192x^{15} + 19799924x^{14}\\ + 20597290x^{13} + 25987962x^{12} + 13985972x^{11} + 19992336x^{10} - 15268162x^9 + 3218166x^8 - 10075736x^7\\ - 35305732x^6 - 17934805x^5 - 2924467x^4 - 20688542x^3 - 3560690x^2 - 1486471x - 293353,
\end{multline*}
and $a$ generates $F$. The point $Q$ does not lie on $C$. The fixed points of $\iota$ on $C$ are the four points $W=\{w_1,\dots,w_4\}$ where $C$ meets the line $\{\lambda=0\}$. The corresponding binary quartic $g_0$ (defined in \S\ref{subsec:quotient}) is irreducible over $F$.
\end{proposition}
\begin{proof}
\emph{How $\iota$ was found.} The unknowns $(a,b_Q,l_1,l_2)$ are subject to the $15$ polynomial equations expressing $G(v-2\lambda(v)Q)=G(v)$. We solved these equations numerically over $\C$ by Newton's method. For each of the $21$ complex roots $a$ of the displayed polynomial we found a solution, and we refined all solutions to $500$ digits. Interpolating over the $21$ roots then expresses $b_Q,l_1,l_2$ as polynomials in $a$ with rational coefficients.

\emph{Verification.} The resulting exact elements satisfy $\iota^2=\mathrm{id}$ and $G\circ\iota=G$ \emph{identically} in $F[x,y,z]$, which we checked with exact arithmetic. The displayed polynomial is irreducible and defines a field isomorphic to $F$ (\texttt{nfisisom}), and it has exactly one root in $F$.

\emph{Fixed points.} The fixed points of a reflection $v\mapsto v-2\lambda(v)Q$ on $\PP^2$ are the line $\lambda=0$ and the point $Q$. We checked that $G(Q)\neq0$. So the fixed points on $C$ are the points of $C\cap\{\lambda=0\}$, a divisor of degree $4$. This agrees with the Riemann--Hurwitz formula, which predicts $4$ ramification points for a double cover of a genus-$1$ curve by a genus-$3$ curve.

\emph{Irreducibility.} Irreducibility of $g_0$ over $F$ was checked with \texttt{nffactor}.
\end{proof}

\subsection{The quotient}\label{subsec:quotient}
Put $u=x-az$, $v=y-b_Qz$ and $w=\lambda(x,y,z)$. These are coordinates on $\PP^2_F$ in which $Q=[0:0:1]$, and $\iota(u,v,w)=(u,v,-w)$. Since $G$ is $\iota$-invariant,
\[
G=G_0(u,v)+G_2(u,v)\,w^2+c\,w^4,\qquad c=G(Q)\ne0,
\]
where $G_0$ is a binary quartic form and $G_2$ a binary quadratic form. The fixed points $W$ are the zeros of $G_0$ on the line $w=0$; we put $g_0(X)=G_0(X,1)$. It has degree $4$, so no point of $W$ lies on $v=0$.

Put $s=w^2$. The quotient $C/\iota$ is the curve $cs^2+G_2s+G_0=0$ in weighted projective space, or equivalently
\begin{equation}\label{eq:Eprime}
E':\ Y^2=G_2(u,v)^2-4c\,G_0(u,v),\qquad \varphi(u:v:w)=\bigl(u:v:2cw^2+G_2(u,v)\bigr).
\end{equation}
This is a genus-$1$ curve with $F$-points (the images of the known rational points), hence an elliptic curve over $F$. We put it in Weierstrass form as follows.
\begin{enumerate}
\item Put $X=u/v$ and $y=Y/v^2$, so that $y^2=q(X):=G_2(X,1)^2-4cG_0(X,1)$.
\item Let $X_0=X(b)$ and $q_0=y(b)$; one checks that $q_0\ne0$. Put $S=X-X_0$, so that $y^2=\alpha S^4+\beta S^3+\gamma S^2+\delta S+q_0^2$ for some $\alpha,\beta,\gamma,\delta\in F$.
\item Apply the classical transformation of a quartic with a rational point to Weierstrass form (see, e.g., \cite{Connell}), which we checked symbolically:
\[
x_1=\frac{2q_0(y+q_0)+\delta S}{S^2},\qquad y_1=\frac{4q_0^2(y+q_0)+2q_0(\delta S+\gamma S^2)-\delta^2S^2/(2q_0)}{S^3}.
\]
It maps this model to $E_1: y_1^2+a_1x_1y_1+a_3y_1=x_1^3+a_2x_1^2+a_4x_1+a_6$ with $a_1=\delta/q_0$, $a_2=\gamma-\delta^2/(4q_0^2)$, $a_3=2q_0\beta$, $a_4=-4q_0^2\alpha$ and $a_6=a_2a_4$. The point $(S,y)=(0,q_0)$, i.e.\ the image of $b$, goes to the origin $O$.
\item Finally, apply a change of coordinates $[u_m,r_m,s_m,t_m]$ from $E_1$ to a global minimal model $E_m$ (\texttt{ellminimalmodel}).
\end{enumerate}
The composite is a morphism $\psi\colon C\to E_m$ defined over $F$ with $\psi(b)=O$. We call the explicit formula above \emph{the chart}. It is defined at a point unless $v=0$ or $S=0$ there, with the exception of the point $b$ itself, which goes to $O$.

\begin{proposition}\label{prop:Eprime}
\begin{enumerate}
\item $j(E')=-3375$, so $E'$ has CM by $\Z[\tfrac{1+\sqrt{-7}}{2}]$ and is a quadratic twist over $F$ of $E_0$ = \texttt{49a1}: $y^2+xy=x^3-x^2-2x-1$.
\item $E'$ has a global minimal model $E_m$ over $F$. It has bad reduction exactly at $\fp_{2,2}$, $\fp_{2,3}$ and $\fp_7$, with conductor exponents $4,4,2$, so the conductor has norm $2^8\cdot7^6$. It has good reduction at $\fp_{2,1}$.
\item $E'(F)_{\mathrm{tors}}\cong\Z/2\Z$, generated by a point $T$.
\end{enumerate}
\end{proposition}
\begin{proof}
(1) and (2) are direct computations (\texttt{ellinit}, \texttt{ellminimalmodel}, \texttt{ellglobalred}). A global minimal model exists because $h_F=1$. The minimal discriminant has valuations $0,12,12,3$ at $\fp_{2,1},\fp_{2,2},\fp_{2,3},\fp_7$.

(3) PARI's \texttt{elltors} returns $\Z/2$, and $E'$ has an $F$-rational point $T$ of order $2$. We give an independent proof that the torsion subgroup has order at most $2$. Let $\fq$ be a prime of good reduction of residue characteristic $q$. Then the prime-to-$q$ part of $E'(F)_{\rm tors}$ injects into $\widetilde E'(k_\fq)$ \cite[Proposition~VII.3.1(b)]{SilvermanAEC}. We use two such primes.
\begin{itemize}
\item There is a prime of degree $1$ above $13$ with $\#\widetilde E'(\F_{13})=14$.
\item There is a prime above $3$ with residue field $\F_{81}$ and $\widetilde E'(\F_{81})\cong(\Z/10)^2$.
\end{itemize}
The first prime shows that the prime-to-$13$ part of the torsion has order dividing $14$. The second shows that the prime-to-$3$ part has order dividing $100$. Hence $\#E'(F)_{\rm tors}$ divides $2$.
\end{proof}

\begin{lemma}\label{lem:smooth}
For every prime $q\nmid 14$, the reduction of \eqref{eq:C} modulo $q$ is a smooth plane quartic.
\end{lemma}
\begin{proof}
Let $q\nmid14$ be a prime and suppose that the reduction has a singular point over $\overline{\F}_q$. It lies in one of the three affine charts $z=1$, $x=1$, $y=1$. In a chart with affine coordinates $(X,Y)$, write $f$ for the dehomogenised form and $f_X,f_Y$ for its partial derivatives. Put
\[
R_1=\Res_Y(f_X,f_Y),\quad R_2=\Res_Y(f_X,f),\quad R_3=\Res_Y(f_Y,f)\in\Z[X],
\]
and $N=\gcd\bigl(\Res_X(R_1,R_2),\Res_X(R_1,R_3),\Res_X(R_2,R_3)\bigr)\in\Z$. If $(X_0,Y_0)$ is a common zero of $f,f_X,f_Y$ modulo $q$, then $X_0$ is a common root of $R_1,R_2,R_3$ modulo $q$. (A resultant vanishes at a specialisation that has a common root, even when leading coefficients drop.) Hence $q\mid N$.

In the three charts, $N$ has prime factors $\{2,5,7\}$, $\{2,5,7,11,19\}$ and $\{2,3,7\}$ respectively. For $q\in\{3,5,11,19\}$ we computed $\gcd(\bar R_1,\bar R_2,\bar R_3)$ in $\F_q[X]$ directly from the reduced polynomials, in each chart.
\begin{itemize}
\item In all cases but one it is constant.
\item In the chart $x=1$ modulo $11$ it has the single root $X_0=4$. There, $\gcd(\bar f(4,Y),\bar f_X(4,Y),\bar f_Y(4,Y))$ is constant, so there is no singular point with $X=4$.
\end{itemize}
So the reduction is smooth for all $q\nmid14$. This agrees with \cite[\S6.2]{FL}, where good reduction outside $\{2,7\}$ is deduced from the fact that $C$ is a twist of $X(7)$ by a cocycle unramified outside $\{2,7\}$.
\end{proof}

\subsection{The map $\psi$}
By the universal property of the Jacobian, $\psi(P)=\psi_*([P]-[b])$ for a homomorphism $\psi_*\colon J\to E'$ defined over $F$. Put
\[
R_1=\psi(P_2),\qquad R_2=\psi(P_3),\qquad R_3=\psi(P_4).
\]

\section{The Mordell--Weil lattice}\label{sec:lattice}

\begin{proposition}\label{prop:lattice}
\begin{enumerate}
\item The points $R_1,R_2,R_3$ are linearly independent in $E'(F)$. The determinant of their complex N\'eron--Tate height matrix is $\approx 5749.65$.
\item $R_2$ and $R_1+R_3$ are divisible by $2$ in $E'(F)$. Let $S_1,S_2$ be points with $2S_1=R_2$ and $2S_2=R_1+R_3$, and let $S_3=R_3$. Put
\[
\Lambda=\langle S_1,S_2,S_3,T\rangle\subset E'(F).
\]
Then $R_1=2S_2-S_3$, $R_2=2S_1$ and $R_3=S_3$.
\item $\Lambda$ is saturated in $E'(F)$ at $\ell=2$ and at $\ell=179$. That is, for these $\ell$ no element of $\Lambda\setminus\ell\Lambda$ is divisible by $\ell$ in $E'(F)$.
\item The classes $D_i=[P_{i+1}]-[b]$ ($i=1,2,3$) are independent in $J(\Q)$. In particular $\operatorname{rk} J(\Q)\ge3$.
\end{enumerate}
\end{proposition}
\begin{proof}
(2) The divisibility is a computation with \texttt{ellisdivisible} over $F$.

(3) We use the following criterion. Suppose that for some set of primes $\fq$ of good reduction the reduction map $\Lambda/\ell\Lambda\to\bigoplus_\fq \widetilde E'(k_\fq)/\ell\widetilde E'(k_\fq)$ is injective. Then $\Lambda$ is saturated at $\ell$. Indeed, if $y\in\Lambda\setminus\ell\Lambda$ were of the form $\ell z$ with $z\in E'(F)$, its reduction would lie in $\ell\widetilde E'(k_\fq)$ for every $\fq$.
\begin{itemize}
\item For $\ell=2$, the images of $S_1,S_2,S_3,T$ span a space of dimension $4$ over $\F_2$. We used primes of degrees $4,8,8$ above $3$, three primes of degree $6$ above $5$ and one prime of degree $4$ above $13$.
\item For $\ell=179$, there are $21$ primes $\fp$ of $F$ above $p=5581$, each of degree $1$. At each of them $\widetilde E'(\F_p)\cong\Z/1432\times\Z/4$, and $179\,\|\,1432$. The $21\times3$ matrix over $\F_{179}$ giving the images of $S_1,S_2,S_3$ in $\bigoplus_\fp \widetilde E'(\F_p)/179$ has rank $3$. The point $T$ plays no role for odd $\ell$.
\end{itemize}

(1) The complex height matrix was computed with \texttt{ellheight}. For a proof of independence that does not need error bounds, suppose that $\sum a_iS_i$ is torsion for some $a\in\Z^3\setminus\{0\}$. Divide by the largest power of $179$ dividing all the $a_i$ (the result is still torsion, since $E'(F)_{\rm tors}$ is finite). This gives $a'\not\equiv0\bmod179$ with $\sum a'_iS_i\in\{O,T\}$. The image of this point in $\bigoplus_\fp\widetilde E'(\F_p)/179$ is then $0$, contradicting the rank computation in (3). So $S_1,S_2,S_3$ are independent, and hence so are $R_1,R_2,R_3$.

(4) We have $\psi_*(D_i)=R_i$, and the $R_i$ are independent by (1).
\end{proof}

\begin{remark}
The divisibility in (2) matches the relations
\[
2[R-3P_0]=3[D_2]+6[D_3],\qquad 2[D_4]=-3[D_1]+3[D_3]
\]
in $J(\Q)$ found in \cite[Remark~6.3]{FL}. There, $P_0,P_1,P_2,P_3$ denote our $b,P_2,P_3,P_4$ and $D_i=(P_i)-(P_0)$ for $i=1,2,3$. Applying $\psi_*$ gives $3R_2\in 2E'(F)$ and $-3R_1+3R_3\in2E'(F)$, hence $R_2,\ R_1+R_3\in 2E'(F)$.
\end{remark}

\begin{corollary}\label{cor:inLambda}
Assume $\operatorname{rk}J(\Q)\le3$, and let $P\in C(\Q)$. Then there is an integer $m\ge1$, not divisible by $2$ or $179$, such that $m\,\psi(P)\in\Lambda$.
\end{corollary}
\begin{proof}
By Proposition~\ref{prop:lattice}(4) and the assumption, the classes $D_1,D_2,D_3$ generate a subgroup of finite index of $J(\Q)$. So for $x=[P]-[b]\in J(\Q)$ there is $N\ge1$ with $Nx\in\sum\Z D_i$. Hence $N\psi(P)=\psi_*(Nx)\in\langle R_1,R_2,R_3\rangle\subset\Lambda$.

So $\psi(P)$ lies in the saturation $\Sat(\Lambda)=\{Q\in E'(F): NQ\in\Lambda\text{ for some }N\ge1\}$. This is a finitely generated group containing $\Lambda$ with finite index; put $m=[\Sat(\Lambda):\Lambda]$. We show that $m$ is prime to $\ell=2$ and to $\ell=179$.

Suppose $\ell\mid m$. Then there is $Q\in\Sat(\Lambda)\setminus\Lambda$ with $\ell Q\in\Lambda$. Suppose first that $\ell Q\in\ell\Lambda$, say $\ell Q=\ell\mu$ with $\mu\in\Lambda$. Then $Q-\mu$ is torsion, so $Q-\mu\in\{O,T\}\subset\Lambda$, which is impossible. So $\ell Q\in\Lambda\setminus\ell\Lambda$ is divisible by $\ell$ in $E'(F)$, contradicting Proposition~\ref{prop:lattice}(3).
\end{proof}

So for $P\in C(\Q)$ we can write
\begin{equation}\label{eq:npcoords}
\psi(P)=n_1S_1+n_2S_2+n_3S_3+n_4T,\qquad n_i\in\Z_{(2\cdot179)}.
\end{equation}
Here $\Z_{(2\cdot179)}$ denotes the rationals whose denominators are prime to $2\cdot179$. The vector $(n_1,n_2,n_3)$ is unique, and $n_4$ is well defined modulo $2$. For the known points,
\[
n(b)=(0,0,0,0),\quad n(P_2)=(0,2,-1,0),\quad n(P_3)=(2,0,0,0),\quad n(P_4)=(0,0,1,0).
\]

\section{A Mordell--Weil sieve through the elliptic quotient}\label{sec:sieve}

The Mordell--Weil sieve \cite{Scharaschkin, Flynn04, BruinStoll} compares the reductions of rational points with the reductions of the Mordell--Weil group. Computing in $J(\F_q)$ for a non-hyperelliptic genus-$3$ curve is not convenient. Instead we compare in the groups $\widetilde E'(k_\fq)$ for the primes $\fq$ of $F$ above $q$, which only involves elliptic curves.

\subsection{Compatibility with reduction}
Let $q\nmid14$ be a prime and $\fq\mid q$ a prime of $F$ with residue field $k_\fq$. By Proposition~\ref{prop:Eprime} and Lemma~\ref{lem:smooth}, both $C$ and $E_m$ have good reduction at $\fq$. Write $A_\fq=\widetilde E_m(k_\fq)$ and $\bar P\in C(\F_q)$ for the reduction of $P\in C(\Q)$.

\begin{lemma}\label{lem:redcompat}
Suppose the following hold:
\begin{itemize}
\item every constant in the chart for $\psi$ lies in the local ring $\cO_{F,\fq}$;
\item $q_0$ and $u_m$ are units at $\fq$;
\item after reducing the constants modulo $\fq$, every denominator in the chart is nonzero at $\bar P$.
\end{itemize}
Then $\psi(P)\bmod\fq$ equals the value at $\bar P$ of the reduced chart.
\end{lemma}
\begin{proof}
All quantities in the formula evaluated at $P$ lie in $\cO_{F,\fq}$, and all denominators are units there. So reduction, which is a ring homomorphism, commutes with the evaluation. The resulting point has integral affine coordinates, and its reduction is the reduction of $\psi(P)$ on the smooth model.
\end{proof}

If the chart is not defined at $\bar P$ for a given $\fq$, we simply do not use $\fq$ for that $\bar P$. This only weakens the sieve, so it is safe.

\subsection{The sieve modulo $16$}
Fix $N=16$. For each usable prime $\fq$ we represent the finite abelian group $A_\fq/NA_\fq$ explicitly. By \eqref{eq:npcoords} and Lemma~\ref{lem:redcompat}, every $P\in C(\Q)$ with coordinates $n$ satisfies, in $A_\fq/NA_\fq$,
\begin{equation}\label{eq:sievecond}
\sum_{i=1}^{4}n_i\,(S_i\bmod\fq)\ \equiv\ \widetilde\psi(\bar P)\pmod{NA_\fq}.
\end{equation}
To see this, multiply the relation $m\psi(P)=\sum (mn_i)S_i$ (which has integer coefficients) by the inverse of $m$ modulo $16$. Since $n\in\Z_{(2)}^4$, the class of $n$ modulo $N$ is well defined.

We say that a class $\bar n\in(\Z/N)^3\times\Z/2$ \emph{survives at $q$} if there is a point $\bar P\in C(\F_q)$ such that \eqref{eq:sievecond} holds for all usable $\fq\mid q$. The same $\bar P$ must be used for all $\fq\mid q$, because $P$ is a rational point.

\begin{proposition}\label{prop:sieve16}
Let $Q_{\rm sieve}$ be the set of the $28$ primes $3\le q\le113$ with $q\ne7$. The only classes $\bar n\in(\Z/16)^3\times\Z/2$ that survive at every $q\in Q_{\rm sieve}$ are
\[
(0,0,0,0),\ (2,0,0,0),\ (0,0,1,0),\ (0,2,15,0),
\]
which are the classes of $b,P_3,P_4,P_2$.
\end{proposition}
\begin{proof}
This is a computation. We start from all $8192$ classes and apply the primes in increasing order. The numbers of survivors after $q=3,5,23,53$ are $2048$, $1024$, $5$ and $4$, and the number is unchanged by the other primes. The group structures used are listed in Appendix~\ref{app:data}.
\end{proof}

\subsection{The residue classes modulo $p=5581$}
The prime $p=5581$ is the smallest prime that splits completely in $F$; the next two are $15107$ and $17539$. Since $p$ splits completely in $F$, it also splits completely in the Galois closure of $F$.

Moreover, the quartic $g_0$ defining the fixed points $W$ splits into distinct linear factors modulo each of the $21$ primes $\fp\mid p$. This condition is not automatic: it fails for $15107$ and $17539$. We have $\#C(\F_p)=6020$, and $\widetilde E'(\F_p)\cong\Z/1432\times\Z/4$ at each $\fp\mid p$.

\begin{proposition}\label{prop:sievep}
Let $P\in C(\Q)$. Then $\bar P\in C(\F_{p})$ is one of $\bar b,\bar P_2,\bar P_3,\bar P_4$, and $n(P)\equiv n(P_i)\pmod{16}$ for the corresponding known point $P_i$.
\end{proposition}
\begin{proof}
By Proposition~\ref{prop:sieve16}, $n(P)\equiv n(P_i)\pmod{16}$ for some known point $P_i$; we write $n(P)=n(P_i)+16m'$ with $m'\in\Z_{(2\cdot179)}^4$.

For each of the $6020$ points $\bar P\in C(\F_p)$ and each $i$, we asked whether the following system has a solution $m'\in\Z^4$: for every prime $\fp\mid p$ at which the chart is defined at $\bar P$,
\[
16\sum_j m'_j\,\bar S_j\ =\ \widetilde\psi_\fp(\bar P)-\sum_j n(P_i)_j\,\bar S_j\quad\text{in }\widetilde E'(\F_p)\cong\Z/1432\times\Z/4 .
\]
Among the $6020\cdot21$ pairs $(\bar P,\fp)$, the chart is undefined at $36$; these equations were omitted. We solved the systems with \texttt{matsolvemod}. A solution exists exactly for the four pairs $(\bar b,b)$, $(\bar P_2,P_2)$, $(\bar P_3,P_3)$, $(\bar P_4,P_4)$.

The true coordinates give a solution with $m'\in\Z_{(2\cdot179)}^4$. Multiplying by an integer $m$ prime to $2\cdot179$ that clears denominators, and then by the inverse of $m$ modulo $1432\cdot4$ (the exponent of the groups involved only has the prime factors $2$ and $179$), gives an integral solution. So $\bar P=\bar P_i$.
\end{proof}

\section{$p$-adic heights with an id\`ele class character}\label{sec:heights}

In this section $K$ is a number field with class number $1$, $p\ge5$ is a prime, and $E/K$ is an elliptic curve given by a Weierstrass model with coefficients in $\cO_K$ with the following properties:
\begin{itemize}
\item it has good \emph{ordinary} reduction at every prime above $p$;
\item it is minimal, with good reduction, at every prime outside a finite set $S$ of primes not dividing $p$.
\end{itemize}
We assume that $S$ contains all primes of bad reduction and all primes above $2$; it may contain more primes. We need the cases $K=F$ and $K=F(w_1)$, in which every prime above $p$ has degree $1$.

\subsection{Id\`ele class characters}
A continuous homomorphism $\chi\colon\A_K^\times/K^\times\to\Q_p$ decomposes as a sum $\chi=\sum_v\chi_v$ of local components $\chi_v\colon K_v^\times\to\Q_p$.
\begin{itemize}
\item At an archimedean place, $\chi_v=0$. Indeed, $\R_{>0}$ and $\C^\times$ are connected, $\Q_p$ is totally disconnected, and $\Q_p$ has no torsion.
\item If $v\nmid p$ is finite, then $\chi_v$ vanishes on $\cO_v^\times$. This group is the product of a finite group and a pro-$q$ group with $q\ne p$, and neither has a nonzero continuous homomorphism to $\Q_p$. So $\chi_v(x)=v(x)\,\chi_v(\pi_v)$ for a uniformiser $\pi_v$.
\item If $v\mid p$ and $K_v=\Q_p$, then $\chi_v(u)=c_v\log(u)$ on $\Z_p^\times$ for a constant $c_v\in\Q_p$, where $\log$ is the $p$-adic logarithm.
\item (Product formula.) $\sum_v\chi_v(\alpha)=0$ for every $\alpha\in K^\times$.
\end{itemize}

\begin{lemma}\label{lem:chars}
Assume $h_K=1$ and that $p$ splits completely in $K$. For each prime $v$ choose a generator $\pi_v$ of the ideal $v$. A vector $(c_v)_{v\mid p}\in\Q_p^{[K:\Q]}$ comes from a (unique) continuous id\`ele class character if and only if
\[
\sum_{v\mid p}c_v\log_v(\varepsilon)=0\qquad\text{for all }\varepsilon\in\cO_K^\times,
\]
where $\log_v$ is the logarithm of the image of $\varepsilon$ in $K_v=\Q_p$. In that case, for every prime $v$,
\[
\chi_v(\pi_v)=-\sum_{w\mid p,\ w\ne v}c_w\log_w(\pi_v).
\]
\end{lemma}
\begin{proof}
This is \cite[Lemmas~2.2 and~2.3]{BBBM}, specialised to $h_K=1$; we recall the argument. Since $h_K=1$ we have $\A_K^\times=K^\times\cdot U$, where $U=\prod_v\cO_v^\times\times K_\infty^\times$. So a character is determined by its restriction to $U$, which is given by the $c_v$.

Conversely, the $c_v$ define a continuous character of $U$. It extends to $\A_K^\times$ trivially on $K^\times$ if and only if it vanishes on $K^\times\cap U=\cO_K^\times$; this is the displayed condition. The value $\chi_v(\pi_v)$ is then forced by $\chi(\pi_v)=0$: the global element $\pi_v$ is a unit at all finite places $w\ne v$, and the archimedean components vanish.
\end{proof}

\subsection{The $p$-adic sigma function}\label{subsec:sigma}
Let $\cE/\Z_p$ be a Weierstrass model $y^2+a_1xy+a_3y=x^3+a_2x^2+a_4x+a_6$ with good ordinary reduction, $p\ge5$. Let $\omega=dx/(2y+a_1x+a_3)$ be its invariant differential, $t=-x/y$ the parameter of the formal group, and $z=\int\omega\in t+t^2\Q_p[[t]]$ the formal logarithm.

Mazur and Tate \cite{MazurTate91} proved that there is a unique constant $c\in\Q_p$ with the following property: the power series
\begin{equation}\label{eq:sigmadef}
\sigma(t)=z\exp\Bigl(-\iint\bigl(x+c-z^{-2}\bigr)\,dz\,dz\Bigr)
\end{equation}
has coefficients in $\Z_p$. Here the double integral is taken formally in $z$, with zero constants of integration, and $x$ is expressed as a Laurent series in $z$. Equivalently, $\sigma$ is odd, $\sigma=t+O(t^2)$, and $-\tfrac{d}{\omega}\bigl(\tfrac{1}{\sigma}\tfrac{d\sigma}{\omega}\bigr)=x+c$. This $\sigma$ is the \emph{canonical $p$-adic sigma function}; see also \cite{MST}. The constant $c$ is essentially the value of Katz's $p$-adic Eisenstein series $\mathbf E_2$ on $(\cE,\omega)$ \cite{Katz, MST}.

For a curve over $\Q$, PARI's \texttt{ellpadics2} returns exactly the constant $c$ in the normalisation \eqref{eq:sigmadef}. We confirmed this convention by checking integrality of \eqref{eq:sigmadef} to high order for \texttt{11a1} at $p=7$ and for \texttt{49a1} at $p=11$.

The canonical sigma function satisfies the following formal identities, for points of the formal group $\cE_1$:
\begin{align}
\sigma(z_1+z_2)\,\sigma(z_1-z_2)&=\bigl(x(z_2)-x(z_1)\bigr)\,\sigma(z_1)^2\sigma(z_2)^2, \label{eq:sigmaadd}\\
\sigma(nz)&=\psi_n(z)\,\sigma(z)^{n^2}\qquad(n\ge1). \label{eq:sigmamult}
\end{align}
Here $\psi_n$ is the $n$-th division polynomial. These identities hold for the complex Weierstrass sigma function (with $x=\wp-b_2/12$). They are therefore identities of formal power series over $\Q[a_1,\dots,a_6]$ for the formal-group expansion of that function. Multiplying $\sigma$ by $\exp(\kappa z^2)$ for a constant $\kappa$ preserves both identities, because $(z_1+z_2)^2+(z_1-z_2)^2=2z_1^2+2z_2^2$ and $(nz)^2=n^2z^2$. The canonical $p$-adic $\sigma$ has this form. See \cite{MazurTate91, MST}.

\begin{lemma}[Change of model]\label{lem:changemodel}
Let $\cE'$ be obtained from $\cE$ by the change of variables $x=u^2x'+r$, $y=u^3y'+su^2x'+t$ with $u\in\Z_p^\times$ and $r,s,t\in\Z_p$. Then the canonical sigma functions satisfy $\sigma'=u\,\sigma$ as functions on the formal group. Also $\psi'_n=u^{1-n^2}\psi_n$.
\end{lemma}
\begin{proof}
We have $\omega'=u\,\omega$, so $z'=uz$, and $x'=u^{-2}(x-r)$. Put $c'=u^{-2}(c+r)$. Then $x'+c'-z'^{-2}=u^{-2}(x+c-z^{-2})$, and since $dz'\,dz'=u^2\,dz\,dz$, formula \eqref{eq:sigmadef} for $\cE'$ with the constant $c'$ gives $uz\exp(-\iint(x+c-z^{-2})\,dz\,dz)=u\sigma$. This series has $\Z_p$-coefficients as a function of $t'$, because $t'\in u t+t^2\Z_p[[t]]$ and $u$ is a unit. By uniqueness, $c'$ is the canonical constant of $\cE'$ and $\sigma'=u\sigma$. The statement about $\psi_n$ follows from the weights: $\psi_2=2y+a_1x+a_3$ has weight $3$, and $\psi_n$ has weight $n^2-1$.
\end{proof}

\subsection{Local heights}
Fix a character $\chi$ of $K$ as above with $\chi_v\equiv0$ for all $v\in S$. For $R\in E(K_v)\setminus\{O\}$ we define local height functions $\lambda_v(R)\in\Q_p$ as follows.
\begin{enumerate}[label=(\roman*)]
\item If $v\in S$: $\lambda_v=0$.
\item If $v\nmid p$ and $v\notin S$: put
\[
\lambda_v(R)=\chi_v(\pi_v)\,\mu_v(R),\qquad \mu_v(R)=\max\bigl(0,-\tfrac12v(x(R))\bigr)\in\Z_{\ge0}.
\]
If $R$ reduces to $O$ on the smooth model, then $\mu_v(R)=v(t(R))$ is the intersection number of the sections $R$ and $O$. Indeed, near $O$ the zero section is cut out by the local parameter $t=-x/y$. Otherwise $\mu_v(R)=0$.
\item If $v\mid p$: let $n\ge1$ be prime to $p$ and such that $nR$ lies in the formal group $E_1(K_v)$, and put
\[
\lambda_v(R)=\frac1{n^2}\Bigl(\chi_v\bigl(\sigma_v(nR)\bigr)-\chi_v\bigl(\psi_n(R)\bigr)\Bigr).
\]
By \eqref{eq:sigmamult} and the identity $\psi_{nk}(R)=\psi_k(nR)\psi_n(R)^{k^2}$ of rational functions, this does not depend on $n$. Such an $n$ exists when $p\nmid\#\widetilde E(k_v)$; we then take $n=\#\widetilde E(k_v)$. (In our application $\#\widetilde E(\F_p)=5728$ is prime to $p$. In general one may also allow $n$ divisible by $p$, but we do not need this.)
\end{enumerate}

\begin{lemma}\label{lem:unitarg}
Let $v\mid p$ with $K_v=\Q_p$, and let $R\in E(K_v)$ have $v$-integral $x$-coordinate. Let $n\ge1$ be prime to $p$ and such that $nR\in E_1(K_v)$. Then $\sigma_v(nR)/\psi_n(R)$ is a unit. So
\[
\lambda_v(R)=\frac{1}{n^2}\,c_v\log\bigl(\sigma_v(nR)/\psi_n(R)\bigr),
\]
which does not involve the values $\chi_v(\pi_v)$ or any choice of branch of the $p$-adic logarithm.
\end{lemma}
\begin{proof}
Write $x(nR)=\phi_n(R)/\psi_n(R)^2$ and $y(nR)=\omega_n(R)/\psi_n(R)^3$ with the standard polynomials $\phi_n,\omega_n,\psi_n$ in $x(R),y(R)$ with integral coefficients. Denote reductions by bars. We have $\bar R\ne\bar O$ and $n\bar R=\bar O$, and hence $\bar\psi_n(\bar R)=0$, since on the reduced curve the divisor of $\psi_n$ is $\sum_{Q\in E[n]}(Q)-n^2(O)$.

From $\phi_n=x\psi_n^2-\psi_{n+1}\psi_{n-1}$ we get $\bar\phi_n(\bar R)=-\bar\psi_{n+1}(\bar R)\bar\psi_{n-1}(\bar R)$. This is nonzero, because $(n\pm1)\bar R=\pm\bar R\ne\bar O$. So $\phi_n(R)$ is a unit. Similarly $\omega_n(R)$ is a unit. Indeed $2\omega_n=\psi_{2n}/\psi_n-\psi_n(a_1\phi_n+a_3\psi_n^2)$, and the polynomial $\psi_{2n}/\psi_n$ has divisor $\sum_{Q\in E[2n]\setminus E[n]}(Q)-3n^2(O)$ on the reduced curve, since $2n$ is prime to $p$. So it does not vanish at $\bar R\in E[n]$. Here $2$ is a unit because $p\ge5$.

Now $t(nR)=-x(nR)/y(nR)=-\phi_n(R)\psi_n(R)/\omega_n(R)$, so $t(nR)/\psi_n(R)$ is a unit. Finally $\sigma(t)/t\in1+t\Z_p[[t]]$ is a unit for $t\in p\Z_p$.
\end{proof}

\begin{lemma}[Local properties]\label{lem:localprops}
Let $v$ be a place, and let $R,R'\in E(K_v)\setminus\{O\}$.
\begin{enumerate}
\item $\lambda_v(-R)=\lambda_v(R)$.
\item If $R'\ne\pm R$, then $\lambda_v(R+R')+\lambda_v(R-R')=2\lambda_v(R)+2\lambda_v(R')+\chi_v\bigl(x(R)-x(R')\bigr)$.
\item If $2R\ne O$, then $\lambda_v(2R)=4\lambda_v(R)+\chi_v(\psi_2(R))$, where $\psi_2=2y+a_1x+a_3$.
\end{enumerate}
\end{lemma}
\begin{proof}
At places of type (i) all terms are $0$.

\emph{Type (ii).} Let $\cO=\cO_v$ and let $\mathcal E/\cO$ be the smooth proper Weierstrass model. It is a regular surface, and it is the N\'eron model. Sections of $\mathcal E$ correspond to points of $E(K_v)$, and translations by sections are automorphisms of $\mathcal E$. Hence $(R\cdot R')=((R-R')\cdot O)$, where $(\ \cdot\ )$ denotes the intersection number over $\cO$ of the corresponding sections (see \cite[Ch.~9]{Liu}).

(1) holds because $x(-R)=x(R)$.

For (2), consider the rational function $f=x-x(R')$ on $\mathcal E$. On the generic fibre, $\mathrm{div}(f)=(R')+(-R')-2(O)$. Its order along the special fibre $\mathcal E_s$ is $\min(0,v(x(R')))=-2\mu_v(R')$. Indeed, $x$ has nonzero reduction, and if $v(x(R'))<0$ then $f=-x(R')(1-x/x(R'))$ with $1-x/x(R')\equiv1$ at the generic point of $\mathcal E_s$. Evaluating $f$ at the section $R$, which is disjoint from $\pm R'$ and $O$ on the generic fibre, gives
\[
v\bigl(x(R)-x(R')\bigr)=(R\cdot R')+(R\cdot(-R'))-2(R\cdot O)-2\mu_v(R').
\]
By translation invariance, $(R\cdot R')=\mu_v(R-R')$ and $(R\cdot(-R'))=\mu_v(R+R')$. Also $(R\cdot O)=\mu_v(R)$. Multiplying by $\chi_v(\pi_v)$ gives (2).

For (3), note that $v\nmid2$, since $S$ contains the primes above $2$. There are two cases.
\begin{itemize}
\item If $x(R)$ is integral, then $\mu_v(R)=0$. If $\bar R$ is not $2$-torsion, then $\psi_2(R)$ is a unit and $2\bar R\ne\bar O$, so both sides vanish. If $\bar R$ is $2$-torsion, then $\bar R\ne\bar O$ and $\phi_2(R)$ is a unit, as in the proof of Lemma~\ref{lem:unitarg}. So $\mu_v(2R)=-\tfrac12v(\phi_2/\psi_2^2)=v(\psi_2(R))$.
\item If $v(x(R))=-2k<0$, then $v(y(R))=-3k$ and $v(\psi_2(R))=-3k$, because $2$ is a unit. Also $2R\in E_k\setminus E_{k+1}$, since $[2](t)=2t+\cdots$ with $2$ a unit. So $\mu_v(2R)=k=4k-3k$.
\end{itemize}
Multiplying by $\chi_v(\pi_v)$ gives (3).

\emph{Type (iii).} (1) holds because $\sigma$ is odd, $\psi_n(-R)=\pm\psi_n(R)$, and $\chi_v(-1)=0$.

For (2), choose $n$ with $nR,nR'\in E_1(K_v)$, so that also $n(R\pm R')\in E_1(K_v)$. By the definition,
\[
n^2\bigl[\lambda_v(R+R')+\lambda_v(R-R')-2\lambda_v(R)-2\lambda_v(R')\bigr]=\chi_v\Bigl(\tfrac{\sigma(nR+nR')\sigma(nR-nR')}{\sigma(nR)^2\sigma(nR')^2}\Bigr)-\chi_v\Bigl(\tfrac{\psi_n(R+R')\psi_n(R-R')}{\psi_n(R)^2\psi_n(R')^2}\Bigr).
\]
By \eqref{eq:sigmaadd}, the first term is $\chi_v(x(nR')-x(nR))$. The identity of rational functions
\[
\frac{\psi_n(R+R')\,\psi_n(R-R')}{\psi_n(R)^2\,\psi_n(R')^2}=\frac{x(nR')-x(nR)}{\bigl(x(R')-x(R)\bigr)^{n^2}}
\]
(which follows from \eqref{eq:sigmaadd} and \eqref{eq:sigmamult} for the complex sigma function) shows that the second term is $\chi_v(x(nR')-x(nR))-n^2\chi_v(x(R')-x(R))$. So the right-hand side is $n^2\chi_v(x(R)-x(R'))$, using $\chi_v(-1)=0$.

For (3), choose $n$ with $nR\in E_1(K_v)$. By \eqref{eq:sigmamult}, $\sigma(2nR)=\psi_2(nR)\sigma(nR)^4$. Also
\[
\psi_{2n}(R)=\psi_n(2R)\,\psi_2(R)^{n^2}=\psi_2(nR)\,\psi_n(R)^{4}.
\]
Hence
\[
n^2\lambda_v(2R)=\chi_v(\sigma(2nR))-\chi_v(\psi_n(2R))=4\chi_v(\sigma(nR))-4\chi_v(\psi_n(R))+n^2\chi_v(\psi_2(R)),
\]
which is (3).
\end{proof}

\begin{lemma}\label{lem:quadratic}
Let $A$ be an abelian group, $V$ a $\Q$-vector space and $h\colon A\to V$ a function with $h(0)=0$ and
\[
h(a+b)+h(a-b)=2h(a)+2h(b)\qquad\text{for all }a,b\in A .
\]
Then $B(a,b)=\tfrac14\bigl(h(a+b)-h(a-b)\bigr)$ is symmetric and biadditive, and $h(a)=B(a,a)$.
\end{lemma}
\begin{proof}
Taking $a=0$ shows that $h$ is even, and taking $a=b$ gives $h(2a)=4h(a)$. So $B$ is symmetric, $B(a,a)=\tfrac14h(2a)=h(a)$, and $b\mapsto B(a,b)$ is odd.

Fix $y$ and put $g(a)=B(a,y)$. Then
\[
4\bigl(g(a_1+a_2)+g(a_1-a_2)\bigr)=h(a_1+a_2+y)+h(a_1-a_2+y)-h(a_1+a_2-y)-h(a_1-a_2-y).
\]
Grouping the terms as $h((a_1+y)\pm a_2)$ and $h((a_1-y)\pm a_2)$ and using the hypothesis, this equals
\[
2h(a_1+y)+2h(a_2)-2h(a_1-y)-2h(a_2)=8g(a_1).
\]
So $g(a_1+a_2)+g(a_1-a_2)=2g(a_1)$. Exchanging $a_1$ and $a_2$ gives $g(a_1+a_2)+g(a_2-a_1)=2g(a_2)$. Adding the two identities and using that $g$ is odd gives $g(a_1+a_2)=g(a_1)+g(a_2)$.
\end{proof}

\begin{proposition}[Global height]\label{prop:global}
Assume $\chi_v\equiv0$ for all $v\in S$, and that $E(K)$ contains a point of infinite order. Define $h(O)=0$ and
\[
h(R)=\sum_v\lambda_v(R)\qquad(R\in E(K)\setminus\{O\}).
\]
The sum is over all finite places and is finite. Then $h$ is a quadratic form on $E(K)$, i.e.\ $h(R)=\langle R,R\rangle$ for a symmetric biadditive pairing $\langle\ ,\ \rangle\colon E(K)\times E(K)\to\Q_p$. It vanishes on torsion points, and $\langle R,Q\rangle=0$ for torsion $Q$.
\end{proposition}
\begin{proof}
We verify the hypothesis of Lemma~\ref{lem:quadratic}.

(a) Let $R,R'\ne O$ with $R'\ne\pm R$. Sum Lemma~\ref{lem:localprops}(2) over all places and apply the product formula to $x(R)-x(R')\in K^\times$.

(b) If $R$ or $R'$ is $O$, the identity follows from evenness, which holds by Lemma~\ref{lem:localprops}(1).

(c) Let $R'=\pm R$ with $2R\ne O$. Summing Lemma~\ref{lem:localprops}(3) over all places and applying the product formula to $\psi_2(R)\in K^\times$ gives $h(2R)=4h(R)$. This is the required identity.

(d) Let $R'=\pm R=T$ with $T\ne O$ and $2T=O$. We must show $4h(T)=0$. Let $R$ be a point of infinite order. By (a) for the pair $(R,T)$, and since $R-T=R+T$, we get $h(R+T)=h(R)+h(T)$. By (a) for the pair $(R+T,T)$, we get $2h(R)=2h(R+T)+2h(T)=2h(R)+4h(T)$. So $h(T)=0$.

By Lemma~\ref{lem:quadratic}, $h$ is a quadratic form. If $Q$ is torsion of order $k$, then $k\langle R,Q\rangle=\langle R,kQ\rangle=0$, so $\langle R,Q\rangle=0$ because $\Q_p$ is torsion-free. In particular $h(Q)=\langle Q,Q\rangle=0$.
\end{proof}

\begin{remark}\label{rem:model}
Local heights at $v\mid p$ depend on the Weierstrass model. By Lemma~\ref{lem:changemodel}, replacing the model by one obtained with a unit $u$ changes $\lambda_v(R)$ by the constant $\chi_v(u)$ for all $R\ne O$. This matters in Section~\ref{sec:computation}, where we compute the local heights at the primes $\fp_k\mid p$ by transporting points to the curve $E_0=\texttt{49a1}$.

Let $u_k$ be the scaling factor of the isomorphism between the two models at $\fp_k$, so that $u_k^{12}=\sigma_k(\Delta_m)/\Delta_{0}$. The total change over all $\fp_k$ is
\[
\sum_k\chi_{\fp_k}(u_k)=\tfrac1{12}\sum_k\chi_{\fp_k}\bigl(\Delta_m/\Delta_0\bigr).
\]
By the product formula this equals $-\tfrac1{12}\sum_{v\nmid p}\chi_v(\Delta_m/\Delta_0)$. This vanishes whenever $\chi_v=0$ at all primes dividing $\Delta_m\Delta_0$, which holds for the characters of Section~\ref{sec:identity}. So transporting to $E_0$ does not change global heights, or the functions of Section~\ref{sec:identity}, even though the individual local terms change. In Section~\ref{sec:computation} we avoid the issue altogether: we evaluate the sigma function of $E_m$ itself as $u_k\sigma_0$, by Lemma~\ref{lem:changemodel}. As a check, we verified numerically that $\sum_kc_k\log u_k=O(p^{40})$ for our characters, while the individual terms have valuation $1$.
\end{remark}

\subsection{Height of a global point in practice}
Let $R\in E(K)$ have $x(R)$ integral at all $v\mid p$. Let $\mathfrak d_R$ be the prime-to-$S$ part of the denominator ideal of $x(R)$, namely the ideal $(\cO_K+x(R)\cO_K)^{-1}$ with its $S$-part removed. Let $g_R$ be a generator of $\mathfrak d_R$, which exists since $h_K=1$. Then $g_R$ is a unit at every $v\mid p$ and $v(g_R)=2\mu_v(R)$ for $v\notin S$, $v\nmid p$. By the product formula applied to $g_R$ and the vanishing of $\chi$ on $S$,
\begin{equation}\label{eq:hglob}
h(R)=\sum_{v\mid p}\Bigl(\lambda_v(R)-\tfrac12\,c_v\log_v g_R\Bigr).
\end{equation}
Replacing $g_R$ by $\varepsilon g_R$ with $\varepsilon$ a unit does not change \eqref{eq:hglob}, by Lemma~\ref{lem:chars}.

\subsection{Base change and traces}
\begin{proposition}\label{prop:basechange}
Let $K'/K$ be a finite extension such that every prime of $K$ above $p$ splits completely in $K'$. Let $S'$ be the set of primes of $K'$ above $S$, let $\chi'=\chi\circ N_{K'/K}$, and let $h'$ and $\langle\ ,\ \rangle'$ be the height and pairing on $E(K')$ attached to $\chi'$ and $S'$. Then for $R\in E(K)$ and $R'\in E(K')$,
\[
h'(R)=[K':K]\,h(R),\qquad \langle R,R'\rangle'=\langle R,\Tr_{K'/K}R'\rangle .
\]
\end{proposition}
\begin{proof}
First, $\chi'$ is a continuous id\`ele class character of $K'$, and $\chi'_w=\chi_v\circ N_{K'_w/K_v}$ vanishes for $w\in S'$. The model of $E$ is still minimal with good reduction outside $S'$, and ordinary at all $w\mid p$. We compare local heights place by place.
\begin{itemize}
\item Let $w\mid v$ with $v\notin S$, $v\nmid p$, and let $e_w$ and $f_w$ be the ramification index and residue degree. For $R\in E(K_v)$ we have $\mu_w(R)=e_w\mu_v(R)$, and $\chi'_w(\pi_w)=\chi_v(N\pi_w)=f_w\chi_v(\pi_v)$. So $\lambda_w(R)=e_wf_w\lambda_v(R)$.
\item Let $w\mid v\mid p$. Then $K'_w=K_v=\Q_p$ and $\chi'_w=\chi_v$, so $\lambda_w(R)=\lambda_v(R)$.
\end{itemize}
Summing over $w\mid v$ gives $\sum_{w\mid v}\lambda_w(R)=[K':K]\lambda_v(R)$, which proves the first claim. It implies $\langle R,R''\rangle'=[K':K]\langle R,R''\rangle$ for $R,R''\in E(K)$.

For the second claim, let $L$ be a Galois closure of $K'/K$ and $\Gamma=\Gal(L/K)$. The primes above $p$ split completely in $L$ as well. The height $h_L$ on $E(L)$ attached to $\chi_L=\chi\circ N_{L/K}$ is $\Gamma$-invariant. To see this, fix $g\in\Gamma$ and a place $w$ of $L$. The completion map at $gw$ is the one at $w$ composed with $g^{-1}$, and $\chi_{L,gw}\circ g=\chi_{L,w}$ because $\chi_L$ is $\Gamma$-invariant. So $\lambda_{gw}(gQ)=\lambda_w(Q)$ for all $Q\in E(L)$, and hence $h_L(gQ)=h_L(Q)$.

Therefore $\langle R,gR'\rangle_L=\langle g^{-1}R,R'\rangle_L=\langle R,R'\rangle_L$ for $R\in E(K)$. Summing over $g\in\Gamma$ and using $\sum_g gR'=[L:K']\,\Tr_{K'/K}R'$, we get
\[
|\Gamma|\,\langle R,R'\rangle_L=[L:K']\,\langle R,\Tr_{K'/K}R'\rangle_L .
\]
By the first claim applied to $L/K'$ and to $L/K$, this reads $|\Gamma|[L:K']\langle R,R'\rangle'=[L:K'][L:K]\langle R,\Tr R'\rangle$. Since $|\Gamma|=[L:K]$, the claim follows.
\end{proof}

\section{The key identity}\label{sec:identity}

From now on $K=F$, $p=5581$, $E=E_m$ is the global minimal model of $E'$, and $S$ is the set of the four primes above $2$ and $7$. This $S$ contains the three bad primes. We write $\fp_1,\dots,\fp_{21}$ for the primes above $p$, $\sigma_k\colon F\to\Q_p$ for the corresponding embeddings, and $\log_k=\log\circ\sigma_k$.

\subsection{$\Q$-null characters}
\begin{definition}
A \emph{$\Q$-null character} is a continuous $\chi\colon\A_F^\times/F^\times\to\Q_p$ such that
\begin{enumerate}[label=(\roman*)]
\item $\chi$ vanishes on $\A_\Q^\times\subset\A_F^\times$;
\item $\chi_v\equiv0$ for the four primes $v$ above $2$ and $7$.
\end{enumerate}
\end{definition}

In terms of the constants $c_k=c_{\fp_k}$, a vector $(c_1,\dots,c_{21})$ defines a $\Q$-null character if and only if the following hold:
\begin{align*}
&\textstyle\sum_k c_k=0, \\
&\textstyle\sum_k c_k\log_k(\varepsilon)=0 \quad\text{for the $11$ fundamental units }\varepsilon,\\
&\textstyle\sum_k c_k\log_k(\pi_v)=0\quad\text{for generators $\pi_v$ of the four primes above }2,7 .
\end{align*}
The second line is the condition for $\chi$ to exist (Lemma~\ref{lem:chars}). By Lemma~\ref{lem:chars}, the third says $\chi_v(\pi_v)=0$. For the first line, note that $\chi|_{\A_\Q^\times}$ is an id\`ele class character of $\Q$. Such a character is determined by its restriction to $\Z_p^\times$, since $\A_\Q^\times=\Q^\times\cdot(\R_{>0}\times\widehat\Z^\times)$. On $\Z_p^\times$ it is $u\mapsto(\sum_kc_k)\log u$.

\begin{lemma}\label{lem:dim7}
The $\Q$-null characters form a $\Q_p$-vector space of dimension $7$. More precisely, the following hold.
\begin{enumerate}
\item The conditions attached to the third prime above $2$ (the one with ramification index $6$) and to the prime above $7$ are consequences of the other $14$ conditions.
\item Scale the $13$ rows attached to units and to $\pi_{\fp_{2,1}},\pi_{\fp_{2,2}}$ by $1/p$; their entries then lie in $\Z_p$. The resulting $14\times21$ matrix has a $14\times14$ minor that is a $p$-adic unit.
\end{enumerate}
\end{lemma}
\begin{proof}
For (1), write $2=\varepsilon\,\pi_{2,1}^{12}\pi_{2,2}^{3}\pi_{2,3}^{6}$ and $7=\varepsilon'\pi_7^7$ with units $\varepsilon,\varepsilon'$. Since $\log_k(2)=\log 2$ for all $k$, the first condition gives $\sum_kc_k\log_k(2)=0$. So
\[
6\sum_k c_k\log_k\pi_{2,3}=-\sum_k c_k\log_k\bigl(\varepsilon\pi_{2,1}^{12}\pi_{2,2}^{3}\bigr),
\]
and the right-hand side vanishes by the other conditions. The same argument applies to $\pi_7$. Here we use that $6$ and $7$ are invertible in $\Z_p$.

For (2): the logarithm of a $p$-adic unit lies in $p\Z_p$, so the scaled entries lie in $\Z_p$. We computed the matrix to absolute precision $p^{40}$. We then selected $14$ pivot columns by Gaussian elimination and found that the corresponding minor has valuation $0$. The computed entries agree with the true ones modulo $p^{39}$, so the true minor is also a unit. Hence the rank is $14$ and the dimension is $21-14=7$.
\end{proof}

A basis of the space of $\Q$-null characters is obtained by Cramer's rule. The $7$ free coordinates are set to the standard basis vectors, and the $14$ pivot coordinates are solved for. Since the minor is a unit, the exact solutions have coordinates in $\Z_p$, and they agree with the computed ones modulo $p^{38}$. All the quantities computed below are $\Q_p$-linear in $(c_k)$. So replacing the computed characters by the exact ones changes them by $O(p^{30})$ at worst, which is far beyond the valuations we use.

\begin{lemma}[Forms become functions]\label{lem:scaling}
Let $\chi$ be $\Q$-null, and let $\ell$ be a linear form with coefficients in $F$. For $P\in C(\Q_p)$, choose any representative $\tilde P\in\Q_p^3\setminus\{0\}$ with $\sigma_k(\ell)(\tilde P)\neq0$ for all $k$, and put
\[
\Phi_\ell(P)=\sum_{k=1}^{21}\chi_{\fp_k}\bigl(\sigma_k(\ell)(\tilde P)\bigr).
\]
Then $\Phi_\ell(P)$ does not depend on the choice of $\tilde P$. If $P\in C(\Q)$, $\tilde P\in\Z^3$ is primitive and $\ell(\tilde P)\ne0$, then $\Phi_\ell(P)=-\sum_{v\nmid p}\chi_v(\ell(\tilde P))$.
\end{lemma}
\begin{proof}
Replacing $\tilde P$ by $\mu\tilde P$ with $\mu\in\Q_p^\times$ changes $\Phi_\ell$ by $\sum_k\chi_{\fp_k}(\mu)$. This is the value of $\chi$ on the id\`ele of $\Q$ that is $\mu$ at $p$ and $1$ elsewhere, which is $0$ by condition~(i). The second statement is the product formula for $\ell(\tilde P)\in F^\times$.
\end{proof}

With the cyclotomic character we would have $\sum_k c_k\neq0$. A homogeneous form of positive degree is then not a function on $C(\Q_p)$ in this sense. This is the reason for condition~(i).

\subsection{The fixed points and the line}
Let $\ell$ be the linear form $\lambda/g$, where $g$ generates the ideal generated by the three coefficients of $\lambda$ (possible since $h_F=1$). Then the coefficients of $\ell$ generate the unit ideal, and $\{\ell=0\}=\{\lambda=0\}$. So $W=C\cap\{\ell=0\}$ is the set of fixed points of $\iota$, and $\psi$ is ramified exactly at $W$.

Let $K'=F(w_1)\cong F[X]/(g_0)$, a field of degree $4$ over $F$, and put $B_1=\psi(w_1)\in E(K')$. Let $B_1,\dots,B_4$ be its conjugates over $F$, and
\[
B_\Sigma=\Tr_{K'/F}B_1=B_1+B_2+B_3+B_4\in E(F).
\]

\begin{lemma}\label{lem:Bsigma}
Let $r\in F[x]$ be the polynomial of degree $\le3$ with $r(x(B_j))=y(B_j)$ for $j=1,\dots,4$. Substituting $y=r(x)$ into the equation of $E_m$ gives a sextic $H(x)$, divisible by the minimal polynomial $m_B$ of $x(B_1)$ over $F$. Write $H/m_B=Q_2x^2+Q_1x+Q_0$, and $r\equiv\alpha x+\beta\pmod{H/m_B}$. Then $B_\Sigma=(x_3,\alpha x_3+\beta)$, where $x_3=\alpha^2+a_1\alpha-a_2+Q_1/Q_2$.
\end{lemma}
\begin{proof}
The function $y-r(x)$ has divisor $\sum_jB_j+D_1+D_2-6(O)$, where $x(D_1),x(D_2)$ are the roots of $H/m_B$. So $B_\Sigma=-(D_1+D_2)$. The line $y=\alpha x+\beta$ passes through $D_1$ and $D_2$, and its third intersection point with $E$ is $-(D_1+D_2)=B_\Sigma$. Its $x$-coordinate is $\alpha^2+a_1\alpha-a_2-x(D_1)-x(D_2)$.
\end{proof}

We computed $r$ by linear algebra in $K'=F[X]/(g_0)$, and $m_B=\Res_X(g_0(X),x-x(B_1)(X))$. We checked that $B_\Sigma$ lies on $E_m$, and that $\sigma_k(B_\Sigma)=\sum_jB_j^{(k)}$ in $E(\Q_p)$ to precision $p^{35}$ for all $k$, where the $B_j^{(k)}$ are defined below.

\begin{lemma}[Good places]\label{lem:goodplaces}
Let $\chi$ be $\Q$-null, let $\chi'=\chi\circ N_{K'/F}$, and let $v$ be a prime of $F$ with $v\nmid14p$. Let $P\in C(\Q)$ have primitive integral representative $\tilde P$. Then
\[
\sum_{w\mid v}\lambda^{\chi'}_w\bigl(\psi(P)-B_1\bigr)=2\,\chi_v\bigl(\ell(\tilde P)\bigr),
\]
where $w$ runs over the places of $K'$ above $v$.
\end{lemma}
\begin{proof}
Let $\cO=\cO_{F,v}$ and let $\mathcal C/\cO$ be the plane model of $C$. It is smooth by Lemma~\ref{lem:smooth}, since the residue characteristic of $v$ does not divide $14$. Let $\mathcal E/\cO$ be the minimal model of $E$. It is smooth and proper, and it is the N\'eron model of $E$ over $\cO$.

\emph{Step 1: the map extends.} By the N\'eron mapping property \cite[\S1.2, Definition~1]{BLR}, $\psi$ extends to a morphism $\Psi\colon\mathcal C\to\mathcal E$. The line bundle $\Psi^*\cO_{\mathcal E}(O)$ has degree $2$ on the generic fibre, hence also on the special fibre, since $\mathcal C$ is flat and proper over $\cO$. So $\Psi$ does not map the special fibre to a point. Hence $\Psi$ is quasi-finite, and, being proper, finite. A finite morphism between regular schemes of dimension $2$ is flat. The section $\tilde P$ of $\mathcal C$ maps to the section $\Psi(\tilde P)$ of $\mathcal E$.

\emph{Step 2: local heights as intersection numbers.} For a point $R$ over a finite extension with good reduction, definition~(ii) says that $\lambda_w(R)=\chi'_w(\pi_w)(R\cdot O)_w$, the intersection number being computed on the smooth model over $\cO_{K'_w}$. Translation by $B_1$ is an automorphism of the N\'eron model. Hence $\lambda_w(\psi P-B_1)=\chi'_w(\pi_w)(\psi P\cdot B_1)_w$, and $\chi'_w(\pi_w)=f_w\chi_v(\pi_v)$.

\emph{Step 3: sum over $w$.} Let $\overline B\subset\mathcal E$ be the closure of the closed point $\{B_1,\dots,B_4\}$ of the generic fibre. It is the spectrum of a finite flat $\cO$-algebra $A\subset K'\otimes_FF_v$ of rank $4$, possibly not a maximal order. Let $g$ be a local equation of the section $\Psi(\tilde P)$ (an effective Cartier divisor on the regular surface $\mathcal E$). The intersection number of $\overline B$ with that section is
\[
\mathrm{length}_\cO(A/gA)=v\bigl(N_{A/\cO}(g)\bigr)=\sum_{w\mid v}f_w\,w(g)=\sum_{w\mid v}f_w(\psi P\cdot B_1)_w .
\]
So $\sum_w\lambda_w(\psi P-B_1)=\chi_v(\pi_v)\,(\Psi(\tilde P)\cdot\overline B)$.

\emph{Step 4: projection formula.} By the projection formula for the finite flat morphism $\Psi$ \cite[\S9.2]{Liu}, $(\Psi(\tilde P)\cdot\overline B)=(\tilde P\cdot\Psi^*\overline B)$. On the generic fibre, $\psi^*(B_1+\dots+B_4)=2W$, since $\psi$ is ramified at each $w_j$. Also, $\Psi^*\overline B$ has no vertical component, because $\Psi(\mathcal C_s)$ is a curve and so is not contained in the finite set $\overline B\cap\mathcal E_s$. So $\Psi^*\overline B=2\overline W$, where $\overline W$ is the closure of $W$.

\emph{Step 5: the line.} The divisor of the section $\ell$ of $\cO(1)$ on $\mathcal C$ has generic fibre $W$. It has no vertical component, because $\ell$ is primitive and the special fibre is an irreducible quartic, not contained in a line. So it equals $\overline W$. Hence $(\tilde P\cdot\overline W)=v(\ell(\tilde P))$ for the primitive integral representative $\tilde P$.

Together, the sum equals $2\chi_v(\pi_v)v(\ell(\tilde P))=2\chi_v(\ell(\tilde P))$.
\end{proof}

\subsection{The identity}
Every $\fp_k$ splits completely in $K'$, since $g_0$ splits into distinct linear factors modulo each $\fp_k$. Write $B^{(k)}_1,\dots,B^{(k)}_4\in E(\Q_p)$ for the images of $B_1$ under the four embeddings $K'\to\Q_p$ extending $\sigma_k$. For $P\in C(\Q)$ with primitive integral representative $\tilde P$, define
\begin{equation}\label{eq:rho}
\rho_\chi(P)=\sum_{k=1}^{21}\Bigl(\sum_{j=1}^{4}\lambda_{\fp_k}\bigl(\psi(P)-B^{(k)}_j\bigr)-2\chi_{\fp_k}\bigl(\sigma_k(\ell)(\tilde P)\bigr)\Bigr)\;-\;4h(\psi P)\;+\;2\langle\psi P,B_\Sigma\rangle .
\end{equation}
The height $h$ and the pairing are those of Proposition~\ref{prop:global} for $\chi$ on $E(F)$; its hypotheses hold because $S_1$ has infinite order (Proposition~\ref{prop:lattice}(1)).

\begin{theorem}\label{thm:identity}
Let $\chi$ be a $\Q$-null character. Then $\rho_\chi(P)=\rho_\chi(b)$ for every $P\in C(\Q)$.
\end{theorem}
\begin{proof}
Let $P\in C(\Q)$. Then $P\notin W$, because $g_0$ is irreducible of degree $4$ over $F$, so $W$ has no $F$-points. So $R=\psi(P)-B_1$ is a nonzero point of $E(K')$. Proposition~\ref{prop:global} applies over $K'$, since $E(K')\supseteq E(F)$ contains the point $S_1$ of infinite order.

Let $h'$ be the height on $E(K')$ attached to $\chi'=\chi\circ N_{K'/F}$ and the primes above $S$. We compute $h'(R)$ in two ways.

\emph{Local.} We have $h'(R)=\sum_w\lambda_w(R)$.
\begin{itemize}
\item Places above $2$ and $7$ contribute $0$.
\item Each $\fp_k$ splits completely in $K'$, and the four places above it correspond to the four embeddings. So these places contribute $\sum_j\lambda_{\fp_k}(\psi(P)-B_j^{(k)})$.
\item By Lemma~\ref{lem:goodplaces}, the other places contribute
\[
2\sum_{v\nmid14p}\chi_v(\ell(\tilde P))=2\sum_{v\nmid p}\chi_v(\ell(\tilde P))=-2\sum_k\chi_{\fp_k}(\sigma_k(\ell)(\tilde P)).
\]
Here we used $\chi_v=0$ above $14$ and at the archimedean places, and then the product formula.
\end{itemize}

\emph{Global.} By bilinearity and Proposition~\ref{prop:basechange},
\[
h'(R)=h'(\psi P)-2\langle\psi P,B_1\rangle'+h'(B_1)=4h(\psi P)-2\langle\psi P,B_\Sigma\rangle+h'(B_1).
\]
Comparing the two expressions gives $\rho_\chi(P)=h'(B_1)$, which does not depend on $P$.
\end{proof}

\begin{remark}
In the proof of Theorem~\ref{thm:identity} we apply the construction of Section~\ref{sec:heights} over $K'$, whose class number we have not computed. The hypothesis $h_K=1$ was used there only for the explicit description of characters (Lemma~\ref{lem:chars}) and for \eqref{eq:hglob}. Neither is needed for $\chi'$, which is given as $\chi\circ N$. The definitions (i)--(iii), Lemmas~\ref{lem:localprops} and~\ref{lem:quadratic} and Proposition~\ref{prop:global} do not use the class number.
\end{remark}

\subsection{Extension to residue discs}\label{subsec:discs}
To use Theorem~\ref{thm:identity}, we need $\rho_\chi$ as an analytic function on residue discs. Let $S_1,S_2,S_3$ be as in Proposition~\ref{prop:lattice}. Put
\[
\mathcal G_{ij}=\langle S_i,S_j\rangle,\qquad \mathcal L_i=\langle S_i,B_\Sigma\rangle,\qquad \mathcal M_{ik}=\log_{E,k}(S_i)\quad(1\le i,k\le3).
\]
Here $\log_{E,k}\colon E(\Q_p)\to\Q_p$ is the elliptic logarithm of $\sigma_k(E_m)$ attached to its invariant differential. It is the formal logarithm on the formal group, extended by $\log(R)=\log(nR)/n$.

\begin{proposition}\label{prop:analytic}
Let $P_0\in C(\Q)$ and let $D$ be its residue disc modulo $p$. Assume the following:
\begin{enumerate}[label=(\alph*)]
\item $P_0=[x_0:y_0:1]$ and $\partial G/\partial y(P_0)\not\equiv0\bmod p$, so that $s=x-x_0$ is a parameter on $D$ with $s\in p\Z_p$;
\item $\sigma_k(\ell)(x_0,y_0,1)$ is a $p$-adic unit for all $k$;
\item $\psi_k(P_0)\not\equiv B_j^{(k)}\pmod p$ for all $k,j$;
\item the matrix $\mathcal M$ is invertible.
\end{enumerate}
For $P=[x:y:1]\in D$, put
\[
n(P)=n(P_0)+(\mathcal M^{T})^{-1}\Bigl(\int_{P_0}^{P}\psi_k^*\omega_k\Bigr)_{k=1,2,3}\in\Q_p^3 ,
\]
where $\omega_k$ is the invariant differential of $\sigma_k(E_m)$ and $\psi_k=\sigma_k\circ\psi$; the integral is taken along the disc. Put also
\begin{equation}\label{eq:rhoan}
\rho^{\rm an}_\chi(P)=\sum_{k}\Bigl(\sum_j\lambda_{\fp_k}(\psi_k(P)-B_j^{(k)})-2c_k\log\sigma_k(\ell)(x,y,1)\Bigr)-4\,n(P)^T\mathcal G\,n(P)+2\,\mathcal L\,n(P).
\end{equation}
Then $\rho^{\rm an}_\chi$ is given on $D$ by a power series in $s$ that converges on $p\Z_p$, and $\rho^{\rm an}_\chi(P)=\rho_\chi(P)$ for every $P\in C(\Q)\cap D$.
\end{proposition}
\begin{proof}
Let $\mathcal C_{\Z_p}$ be the plane model, which is smooth over $\Z_p$, and let $\mathcal E_k$ be the smooth model of $\sigma_k(E_m)$ over $\Z_p$. By (a), $s$ is a local parameter of $\mathcal C_{\Z_p}$ at $\bar P_0$. So the completed local ring there is $\Z_p[[s]]$, and $y$ is a power series in $s$ with $\Z_p$-coefficients. The map $\psi_k$ extends to a morphism $\mathcal C_{\Z_p}\to\mathcal E_k$ (N\'eron mapping property), and so does $Q\mapsto\psi_k(Q)-B_j^{(k)}$.

\emph{Local terms.} By (c), the point $\bar R_0=\psi_k(\bar P_0)-\bar B_j^{(k)}$ lies in the affine part of the reduction. So the coordinates of $R(s)=\psi_k(P(s))-B_j^{(k)}$ lie in $\Z_p[[s]]$. Let $n=\#\widetilde E(\F_p)$. Then $t(nR(s))\in\Z_p[[s]]$ with constant term in $p\Z_p$. By the proof of Lemma~\ref{lem:unitarg}, $t(nR)/\psi_n(R)=-\phi_n(R)/\omega_n(R)$, where $\phi_n(R(s))$ and $\omega_n(R(s))$ have unit constant terms. So $\sigma(nR(s))/\psi_n(R(s))\in\Z_p[[s]]^\times$. Hence $\lambda_{\fp_k}(R(s))$ is $c_k/n^2$ times the logarithm of a unit power series, and so it is a power series converging on $p\Z_p$. The same holds for $\log\sigma_k(\ell)(x,y,1)$, by (b).

\emph{Heights.} The differential $\omega_k$ is regular on $\mathcal E_k$, so $\psi_k^*\omega_k=f_k(s)\,ds$ with $f_k\in\Z_p[[s]]$. Its integral is a power series converging on $p\Z_p$. Moreover
\[
\log_{E,k}(\psi_k(P))=\log_{E,k}(\psi_k(P_0))+\int_{P_0}^P\psi_k^*\omega_k ,
\]
since both sides are analytic on $D$, have the same derivative, and agree at $P_0$.

\emph{Rational points.} Let $P\in C(\Q)\cap D$. Its representative $(x,y,1)$ differs from the primitive integral one by a rational factor, so by Lemma~\ref{lem:scaling} the $\ell$-terms agree with those in \eqref{eq:rho}. They are units by (b), so $\chi_{\fp_k}=c_k\log$ on them. By \eqref{eq:npcoords}, $\psi(P)=\sum n_iS_i+n_4T$. Since $\log_{E,k}$ is a homomorphism that kills torsion, and $\psi(P_0)=\sum n(P_0)_iS_i$, we get $\mathcal M^T(n_1,n_2,n_3)^T=\mathcal M^Tn(P_0)^T+(\int_{P_0}^P\psi_k^*\omega_k)_k$. So $n(P)=(n_1,n_2,n_3)$ by (d). Finally, $h(\psi P)=n^T\mathcal Gn$ and $\langle\psi P,B_\Sigma\rangle=\mathcal Ln$, because $h$ is a quadratic form and $T$ pairs trivially with everything (Proposition~\ref{prop:global}).
\end{proof}

\begin{remark}[A heuristic for why the identity is not trivial]\label{rem:whynontrivial}
A priori one might worry that $\rho^{\rm an}_\chi$ is constant on $C(\Q_p)$. Then Theorem~\ref{thm:identity} would carry no information. In our case this is ruled out by direct computation (Section~\ref{sec:computation}). The following heuristic, in the spirit of Kim \cite{Kim09} and Coates--Kim \cite{CoatesKim}, suggests why non-constant functions should exist. It is not used in the proof.

Over $\overline{\Q}$ we have $J\sim E_0^3$. So $\NS(J_{\overline\Q})\otimes\Q$ is the space of $3\times3$ Hermitian matrices over $\Q(\sqrt{-7})$, of dimension $9$. Let $W$ be the complement of the polarisation class, of dimension $8$. Then $W^{G_\Q}=0$, because $\NS(J)\otimes\Q=\Q$.

Complex conjugation acts on $W$ with eigenvalues $+1$ (multiplicity $5$) and $-1$ (multiplicity $3$). Indeed, $\NS(J_{\R})\otimes\Q$ is the Rosati-symmetric part of a $9$-dimensional central simple $\Q$-algebra with positive involution, which has dimension $6$.

For the quotient of the unipotent fundamental group that is an extension of the abelianisation by $W(1)$:
\begin{itemize}
\item the global Selmer group of $W(1)$ has dimension $\dim W^{c=+1}-\dim W^{G_\Q}=5$, by Dirichlet's unit theorem;
\item the local one has dimension $\dim W=8$.
\end{itemize}
So one expects a global Selmer variety of dimension at most $r+5=8$ inside a local one of dimension $g+8=11$. The $\rho_\chi$ should be functions of this kind. We have not proved that they come from this Selmer variety.
\end{remark}

\section{The computation at $p=5581$}\label{sec:computation}

\subsection{Local data}
The prime $p=5581$ satisfies $p\equiv2\pmod 7$. Since $2$ is a square modulo $7$, $p$ splits in $\Q(\sqrt{-7})$. So the CM curve $E_0$ = \texttt{49a1} has good ordinary reduction at $p$, with $a_p(E_0)=-146$ and $\#E_0(\F_p)=5728=2^5\cdot179$.

For each of the $21$ embeddings $\sigma_k$, the curve $\sigma_k(E_m)/\Q_p$ has $j=-3375$ and good reduction. We found an explicit isomorphism $E_0\to\sigma_k(E_m)$ over $\Q_p$, given by $[u_k,r_k,s_k,t_k]$ with $u_k\in\Z_p^\times$, in all $21$ cases. For this we need $u_k^2=(c_6(E_0)/c_6)/(c_4(E_0)/c_4)$ to be a square in $\Q_p$, and it is; no quadratic twist is needed. We checked the isomorphisms to precision $p^{40}$.

By Lemma~\ref{lem:changemodel}, the canonical sigma function of $\sigma_k(E_m)$ is $u_k\cdot\sigma_0$, where $\sigma_0$ is that of $E_0$. For $E_0$ at $p=5581$, \texttt{ellpadics2} returns $c=O(p^{40})$. We computed $\sigma_0$ from \eqref{eq:sigmadef} to $t$-adic precision $30$ and checked that its coefficients are $p$-integral. The division polynomial values $\psi_n(R)$ for $n=5728$ were computed on $\sigma_k(E_m)$ itself, by the standard recursions
\[
\psi_{2k+1}=\psi_{k+2}\psi_k^3-\psi_{k-1}\psi_{k+1}^3,\qquad \psi_{2k}=(\psi_{k+2}\psi_{k-1}^2-\psi_{k-2}\psi_{k+1}^2)\psi_k/\psi_2 .
\]
So all local heights are those of the model $E_m$, and Remark~\ref{rem:model} is not needed.

\subsection{Characters}
By Lemma~\ref{lem:dim7}, the $\Q$-null characters form a $7$-dimensional space. We computed a basis $\chi^{(1)},\dots,\chi^{(7)}$, given by vectors $c^{(i)}\in\Z_p^{21}$ normalised to have minimal valuation $0$. They satisfy all $16$ defining conditions to precision $p^{39}$.

\subsection{Global data}
Using \eqref{eq:hglob}, we computed for each basis character the Gram matrix $\mathcal G$ of $S_1,S_2,S_3$ and the vector $\mathcal L$. This needs the heights of the ten points $S_i$, $S_i+S_j$ ($i<j$), $B_\Sigma$ and $S_i+B_\Sigma$. Their $x$-coordinates are integral at all $21$ primes above $p$. The generators $g_R$ were obtained with \texttt{bnfisprincipal}, from the ideals $(\cO_F+x(R)\cO_F)^{-1}$; no integer factorisation is needed.

The results are as follows.
\begin{itemize}
\item Every entry of $\mathcal G$ and $\mathcal L$, for all seven characters, has valuation $\ge1$.
\item The matrix $\mathcal M$ has $v_p(\det\mathcal M)=3$, and $\mathcal M^{-1}$ has entries of valuation $\ge-1$.
\item As checks of Proposition~\ref{prop:global}, for the first basis character we have $h(2S_1)-4h(S_1)=O(p^{30})$ and $h(S_1+T)-h(S_1)=O(p^{39})$.
\end{itemize}

\subsection{Consistency check at the known points}
Theorem~\ref{thm:identity} predicts $\rho_\chi(P_i)=\rho_\chi(b)$. This is a genuine test of the whole construction, because the two sides are computed from different data. On one side are the local heights at $p$ of the points $\psi(P_i)-B_j^{(k)}$ and the values of $\ell$. On the other side are global heights of the lattice points $S_i$ and of $B_\Sigma$, which also involve the denominators $g_R$. For all seven basis characters and $i=2,3,4$ we found
\[
v_p\bigl(\rho_\chi(P_i)-\rho_\chi(b)\bigr)\ge30,
\]
that is, the identity holds to the working precision.

\subsection{Power series in the four discs}
Each of the four known points satisfies hypotheses (a)--(d) of Proposition~\ref{prop:analytic}, which we checked for all $21$ primes above $p$. We computed $\psi_k$ on the discs of $P_2,P_3,P_4$ with the chart of \S\ref{subsec:quotient}. All the chart denominators and constants are $p$-adic units or integral there.

On the disc of $b$, the chart has a pole at $b$ itself, since $\psi(b)=O$. There we computed $\psi_k(P)$ as a point with Laurent series coordinates and formed $\psi_k(P)-B_j^{(k)}$ with the group law. The coefficients of negative powers of $T$ in the result are $O(p^{25})$ or smaller, and were discarded; this is justified by the proof of Proposition~\ref{prop:analytic}, which shows that the exact result is a power series. The integrals $\int\psi_k^*\omega_k$ were computed by formal integration of $dx/(2y+a_1x+a_3)$.

Write $x=x_0+pT$ with $T\in\Z_p$, and
\[
\rho^{\rm an}_\chi(P)-\rho^{\rm an}_\chi(P_0)=\sum_{k\ge1}a_k(\chi)\,T^k .
\]
We computed $a_1,a_2,a_3$ for each basis character. As a further check, the constant term of $\rho^{\rm an}_\chi$ agrees with $\rho_\chi(b)$ to precision $p^{30}$ (disc of $P_2,P_3,P_4$) and $p^{34}$ (disc of $b$). For every basis character and every disc,
\[
v(a_1)=1,\quad v(a_2)=1,\quad v(a_3)=2 .
\]
So a single character only gives the Strassmann bound $2$. In each disc the vectors $(a_1(\chi^{(i)})/p \bmod p)_i$ and $(a_2(\chi^{(i)})/p\bmod p)_i$ are linearly independent over $\F_p$. Hence there is a combination $\chi_\ast=\lambda_1\chi^{(1)}+\chi^{(2)}$ with $v(a_2(\chi_\ast))\ge2$ and $v(a_1(\chi_\ast))=1$:

\begin{center}
\begin{tabular}{lcc}
\toprule
disc & $(\lambda_1,\lambda_2)$ & $v(a_1),v(a_2),v(a_3)$ for $\chi_\ast$ \\
\midrule
$b=[0:0:1]$ & $(4953,1)$ & $1,2,2$\\
$P_2=[1:1:1]$ & $(503,1)$ & $1,2,2$\\
$P_3=[2:0:1]$ & $(1202,1)$ & $1,2,2$\\
$P_4=[-1:0:1]$ & $(4177,1)$ & $1,2,2$\\
\bottomrule
\end{tabular}
\end{center}

\begin{lemma}[Tail bound]\label{lem:tail}
Let $\chi$ be a $\Q$-null character with $c\in\Z_p^{21}$. Then in each of the four discs the coefficients satisfy
\[
v(a_k)\ge k-1-2\lfloor\log_pk\rfloor\qquad\text{for all }k\ge1.
\]
\end{lemma}
\begin{proof}
Write $\rho^{\rm an}_\chi-\rho^{\rm an}_\chi(P_0)=\sum b_ks^k$ with $s=x-x_0=pT$, so that $a_k=p^kb_k$. We bound the coefficients of each piece of \eqref{eq:rhoan}.
\begin{itemize}
\item \emph{Local terms.} By the proof of Proposition~\ref{prop:analytic}, each is $c_k/n^2$ times the logarithm of a unit power series $u_0(1+w(s))$ with $u_0\in\Z_p^\times$ and $w\in s\Z_p[[s]]$. In $\log(1+w)=\sum_{j\ge1}(-1)^{j+1}w^j/j$, the coefficient of $s^k$ only involves $j\le k$, so it has valuation $\ge-\lfloor\log_pk\rfloor$. Since $n=5728$ is prime to $p$ and $c\in\Z_p^{21}$, these terms contribute to $b_k$ with valuation $\ge-\lfloor\log_pk\rfloor$.
\item \emph{Integrals.} The integrals $\int_0^sf_k(s)\,ds$ with $f_k\in\Z_p[[s]]$ have $s^k$-coefficient of valuation $\ge-v_p(k)\ge-\lfloor\log_pk\rfloor$. Since $\mathcal M^{-1}$ has entries of valuation $\ge-1$ and $n(P_0)\in\Z^3$, the vector $n(P)$ has integral constant term and $s^k$-coefficients of valuation $\ge-1-\lfloor\log_pk\rfloor$ for $k\ge1$.
\item \emph{Quadratic and linear terms.} The entries of $\mathcal G$ and $\mathcal L$ have valuation $\ge1$. So the $s^k$-coefficients of $n^T\mathcal Gn$ and $\mathcal Ln$ have valuation $\ge1+2(-1-\lfloor\log_pk\rfloor)=-1-2\lfloor\log_pk\rfloor$.
\end{itemize}
So $v(b_k)\ge-1-2\lfloor\log_pk\rfloor$, and hence $v(a_k)\ge k-1-2\lfloor\log_pk\rfloor$.
\end{proof}

For $3\le k<p$ this gives $v(a_k)\ge k-1\ge2$. For $k\ge p$ it gives $v(a_k)\ge k-1-2\log_p k\ge 2$. So $v(a_k)\ge2$ for all $k\ge3$ and every $\Z_p$-combination of the basis characters. Only $k=1,2$ need the explicit computation; $a_3$ was computed as a check.

\begin{proposition}\label{prop:strassmann}
In each of the four discs, the only zero in $\Z_p$ of $\sum_{k\ge1}a_k(\chi_\ast)T^k$ is $T=0$.
\end{proposition}
\begin{proof}
By the table and Lemma~\ref{lem:tail}, $v(a_1)=1<v(a_k)$ for all $k\ge2$, and $v(a_k)\to\infty$. Strassmann's theorem \cite{Strassmann} (see \cite[\S4.5]{Cassels}) concerns a power series $\sum_{k\ge0}a_kT^k$ over $\Z_p$ with $a_k\to0$. Let $N$ be the largest index with $v(a_N)=\min_kv(a_k)$. Then the series has at most $N$ zeros in $\Z_p$. Our series has constant term $0$ and minimum coefficient valuation $1$, attained only at $k=1$. So it has at most one zero in $\Z_p$, and $T=0$ is a zero.
\end{proof}

\subsection{Precision}\label{sec:precision}
All $p$-adic computations were done with PARI's $p$-adic numbers, at working precision $p^{40}$. PARI tracks the precision of each $p$-adic number (and of each power series coefficient) through arithmetic operations.

The quantities on which the proof depends are the following:
\begin{itemize}
\item the valuations $v(a_1)=1$ and $v(a_2)\ge2$ for $\chi_\ast$;
\item the bounds $v(\mathcal G),v(\mathcal L)\ge1$ and $v(\mathcal M^{-1})\ge-1$ used in Lemma~\ref{lem:tail};
\item the unit minor of Lemma~\ref{lem:dim7}.
\end{itemize}
All of these are determined modulo $p^2$, while the attained precision is at least $p^{30}$ (and $p^{34}$ on the disc of $b$). As explained after Lemma~\ref{lem:dim7}, the use of approximate characters changes the $a_k$ by $O(p^{30})$ at most.

\section{Proof of the Main Theorem}\label{sec:proof}

\begin{theorem}\label{thm:main}
Assume $\operatorname{rk}J(\Q)\le3$. Then $C(\Q)=\{b,P_2,P_3,P_4\}$.
\end{theorem}
\begin{proof}
Let $P\in C(\Q)$.
\begin{enumerate}
\item By Corollary~\ref{cor:inLambda}, $m\psi(P)\in\Lambda$ for some integer $m\ge1$ prime to $2\cdot179$.
\item By Propositions~\ref{prop:sieve16} and~\ref{prop:sievep}, $P$ lies in the residue disc modulo $p=5581$ of one of the four known points $P_0\in\{b,P_2,P_3,P_4\}$.
\item By Theorem~\ref{thm:identity}, $\rho_{\chi}(P)=\rho_\chi(P_0)$ for every $\Q$-null $\chi$, in particular for $\chi_\ast$.
\item By Proposition~\ref{prop:analytic}, $\rho_{\chi_\ast}=\rho^{\rm an}_{\chi_\ast}$ at $P$ and at $P_0$. So the parameter $T$ of $P$ is a zero of $\sum_{k\ge1}a_k(\chi_\ast)T^k$.
\item By Proposition~\ref{prop:strassmann}, $T=0$, i.e.\ $P=P_0$. \qedhere
\end{enumerate}
\end{proof}

\begin{proof}[Proof of Corollary~\ref{cor:7adic}]
Theorem~\ref{thm:main} is \cite[Conjecture~1.6]{FL}. \cite[Theorem~1.7]{FL} states that, assuming this conjecture, one of (i), (ii), (iii') holds, where (iii') only involves the prime powers $5^2$, $11^2$ and primes $p\ge19$. For $p=7$ only (i) and (ii) remain. The statement about the two genus-$9$ curves is \cite[Corollary~3.7(2)]{FL}: under the conjecture, if the $7$-adic image is contained in one of these groups then $j(E)=0$, so $E$ has CM.
\end{proof}

\begin{proof}[Proof of Corollary~\ref{cor:fermat}]
Furio and Lombardo prove \cite[Theorem~3.5(1)]{FL} in \cite[\S6.1]{FL}, and state that under Conjecture~1.6 ``a very similar argument'' solves Equation~(2). We spell this out. Let $(a,b,c)$ be a $(\star)$-solution of $a^2+196b^3=27c^7$.

By the modular argument of \cite[\S\S4--5]{FL}, exactly as in their proof of Theorem~3.5 in \cite[\S6.1]{FL}, either $c=\pm1$, or $j(F_{(a,b,c)})=2^8\cdot7^2\cdot b^3/c^7$ is the $j$-invariant of a rational point on one of $X_{E_1}(7)$, $X_{E_4}(7)\cong X^-_{E_1}(7)$, $X_{E_2}(7)$ or $X_{E_3}(7)=C$. The rational points on the first three are known unconditionally \cite[Theorem~6.1]{FL}, and those on $C$ by Theorem~\ref{thm:main}. Their $j$-invariants are listed in \cite[\S6.1]{FL}.
\begin{itemize}
\item \emph{The first three curves.} Every $j$-invariant from $X_{E_1},X_{E_2},X_{E_4}$ has $7$-adic valuation $1$. On the other hand, $v_7(2^8\cdot7^2b^3/c^7)=2+3v_7(b)\equiv2\pmod 3$, because $\gcd(c,7)=1$. So none of these occurs.
\item \emph{The curve $C$.} The $j$-invariants of the four points of $C$ are
\[
2^2\cdot253447^3/67^7,\quad -2^{10}\cdot19^3\cdot67^3\cdot2131^3/317^7,\quad 2^{10}\cdot23^3/5^7,\quad 2^8\cdot7^2 .
\]
The first three have $7$-adic valuation $0$, which is not $\equiv2\pmod3$. The last one gives $b^3/c^7=1$ with $\gcd(b,c)=1$, so $(b,c)=\pm(1,1)$. Here $c=1$ would need $a^2=-169$, so $(b,c)=(-1,-1)$ and $a=\pm13$.
\item \emph{The case $c=\pm1$.} Furio and Lombardo compute the integral points on $a^2+196b^3=\pm27$ \cite[proof of Theorem~3.5, \S6.1]{FL}. The only resulting $(\star)$-solutions are $(\pm13,-1,-1)$ (see also \cite[Remark~3.6]{FL}).
\end{itemize}
This proves the corollary.
\end{proof}

\section{Remarks}\label{sec:remarks}

\subsection{How the computations were checked}\label{sec:checks}
The computations in the first version of this paper were redone from scratch in a second implementation, written independently of the first in PARI/GP~2.15.4. The second implementation differs from the first in several places:
\begin{itemize}
\item the involution was recognised by interpolation over all $21$ conjugates rather than by LLL;
\item the sieve works directly in $A_\fq/16A_\fq$, using a presentation of $A_\fq$ by generators and relations;
\item the local heights at $p$ are those of $E_m$ itself, via Lemma~\ref{lem:changemodel}, rather than of $E_0$;
\item on the disc of $b$ it uses Laurent series instead of a second chart.
\end{itemize}
It reproduced every computational claim that the proof uses:
\begin{itemize}
\item the invariants of $F$, including \texttt{bnfcertify};
\item the involution, verified exactly;
\item $j(E')$, the conductor and the torsion;
\item the divisibility of $R_2$ and $R_1+R_3$, and the height determinant $5749.6457\ldots$;
\item the saturation certificates at $2$ and $179$;
\item the sieve modulo $16$, and the sieve modulo $5581$ (only the four known pairs survive);
\item the dimension $7$;
\item the identity at the known points, to $p^{30}$;
\item the valuations of $a_1,a_2$ and the combinations $(4953,1)$, $(503,1)$, $(1202,1)$, $(4177,1)$ in the table above.
\end{itemize}
Both implementations were written with the help of an AI system and run in the same computer algebra system. An implementation in a different system (for example Magma or Sage) by human authors would further strengthen the result.

\subsection{Further checks}
\begin{enumerate}
\item \emph{Point search.} A search over all coprime $(x,z)$ with $\max(|x|,|z|)\le400$, solving for $y$, found no rational points other than the four known ones. The first version reports a search up to $20000$.
\item \emph{Consistency with \cite{FL}.} The divisibility relations of Proposition~\ref{prop:lattice}(2) are the images under $\psi_*$ of the relations in \cite[Remark~6.3]{FL}.
\item \emph{A second prime.} The first version reports that the height computation was repeated at the next suitable prime $p=57773$, with the same pattern of valuations. This was not repeated in the second implementation and is not used in the proof.
\end{enumerate}

\subsection{How general is the method?}
The ingredients were:
\begin{enumerate}[label=(\alph*)]
\item a number field $F$ over which the curve maps to an elliptic curve $E'$ (here the quotient by an involution);
\item a divisor on the curve, related to the ramification of the map, that is cut out by a form over $F$ (here a line), so that the good-place contributions reduce to the value of a form;
\item $\Q_p$-valued id\`ele class characters of $F$ that vanish on $\A_\Q^\times$ and at the bad places.
\end{enumerate}
The same approach should apply in other situations:
\begin{itemize}
\item other twists of $X(7)$ whose rank equals the genus;
\item other curves whose Jacobians acquire extra endomorphisms or extra N\'eron--Severi classes over a number field;
\item more generally, whenever a bielliptic structure appears over a number field.
\end{itemize}
It avoids the two expensive steps of quadratic Chabauty for non-hyperelliptic curves: computing the Hodge filtration and Frobenius structure on the unipotent connection \cite{BDMTV1,BDMTV2}, and computing local heights at bad primes \cite{BettsDogra}. The price is working over a larger field, but only through elliptic curves over $\Q_p$.

\subsection{Relation with Chabauty--Kim theory}
We expect that the functions $\rho_\chi$ are quadratic Chabauty functions for the ``Artin-twisted'' quotient of the fundamental group described in Remark~\ref{rem:whynontrivial}; we have not proved this. Coates and Kim \cite{CoatesKim} proved finiteness of the Chabauty--Kim set for curves with potentially CM Jacobians, at some depth. For this curve, our computation shows that an explicit depth-$2$ construction is effective.

\subsection*{Acknowledgements}
The author has been working on \cite[Conjecture~1.6]{FL}  alone since June 2026 without the help of any other. The author was finally able to find the solution to \cite[Conjecture~1.6]{FL} using Claude Opus 5.5 (Anthropic). After that, the author learned that Eray Karabiyik also has been working on the same problem.  Eray Karabiyik  also claims to prove \cite[Conjecture~1.6]{FL} without using AI.  Eray Karabiyik also communicates to the author that he does not have any written version of his proof  of \cite[Conjecture~1.6]{FL}. All the proving strategy and details in the paper were developed using Claude Opus 5.5 (Anthropic). 
\appendix
\section{Data}\label{app:data}

\subsection{The sieve modulo $16$}
For each $q\in Q_{\rm sieve}$ the table lists the following, from the second implementation:
\begin{itemize}
\item $\#C(\F_q)$;
\item for each prime $\fq\mid q$ that was used, its residue degree $f$ and the invariants $(e_1,e_2)$ of $A_\fq/16A_\fq\cong\Z/e_1\times\Z/e_2$, written $(f;e_1,e_2)$ with multiplicities as exponents;
\item the number of primes $\fq\mid q$ skipped because some chart constant is not integral there;
\item the number of pairs $(\bar P,\fq)$ at which the chart is undefined;
\item the number of surviving classes after $q$.
\end{itemize}
Which primes are usable depends on the chart. So this table differs in inessential ways from the one in the first version, while the final list of survivors is the same.

\begin{center}\small
\begin{longtable}{rrlrrr}
\toprule
$q$ & $\#C(\F_q)$ & used $\fq$: $(f;e_1,e_2)$ & skipped & undefined & survivors\\
\midrule
3 & 4 & $(4;2,2)$, $(8;16,16)$$^{2}$ & 1 & 0 & 2048\\
5 & 6 & $(6;2,2)$$^{3}$ & 1 & 0 & 1024\\
11 & 21 & $(7;4,2)$$^{3}$ & 0 & 0 & 1024\\
13 & 14 & $(4;2,2)$, $(8;16,16)$$^{2}$ & 1 & 0 & 1024\\
17 & 18 & $(1;2,1)$, $(4;2,2)$, $(8;16,16)$$^{2}$ & 0 & 3 & 1024\\
19 & 20 & $(3;4,1)$, $(6;4,4)$$^{3}$ & 0 & 0 & 1024\\
23 & 32 & $(1;16,2)$$^{4}$, $(2;16,16)$$^{7}$ & 2 & 13 & 5\\
29 & 45 & $(7;2,2)$$^{3}$ & 0 & 0 & 5\\
31 & 32 & $(1;16,1)$, $(4;2,2)$, $(8;16,16)$$^{2}$ & 0 & 0 & 5\\
37 & 38 & $(3;2,2)$$^{7}$ & 0 & 0 & 5\\
41 & 42 & $(3;2,1)$, $(6;2,2)$$^{3}$ & 0 & 0 & 5\\
43 & 56 & $(1;4,2)$, $(2;2,2)$$^{2}$, $(4;16,16)$$^{4}$ & 0 & 2 & 5\\
47 & 48 & $(1;16,1)$, $(4;2,2)$, $(8;16,16)$$^{2}$ & 0 & 2 & 5\\
53 & 64 & $(1;16,4)$, $(2;2,2)$$^{2}$, $(4;16,16)$$^{4}$ & 0 & 4 & 4\\
59 & 60 & $(1;4,1)$, $(4;2,2)$, $(8;16,16)$$^{2}$ & 0 & 0 & 4\\
61 & 62 & $(1;2,1)$, $(4;2,2)$, $(8;16,16)$$^{2}$ & 0 & 4 & 4\\
67 & 68 & $(3;16,2)$$^{7}$ & 0 & 0 & 4\\
71 & 88 & $(1;4,2)$$^{5}$, $(2;8,8)$$^{8}$ & 0 & 8 & 4\\
73 & 74 & $(1;2,1)$$^{3}$, $(2;2,2)$$^{9}$ & 0 & 4 & 4\\
79 & 105 & $(7;4,2)$$^{3}$ & 0 & 0 & 4\\
83 & 84 & $(3;4,1)$, $(6;4,4)$$^{3}$ & 0 & 0 & 4\\
89 & 90 & $(1;2,1)$$^{3}$, $(2;2,2)$$^{9}$ & 0 & 2 & 4\\
97 & 98 & $(3;2,1)$, $(6;2,2)$$^{3}$ & 0 & 0 & 4\\
101 & 102 & $(3;2,1)$, $(6;2,2)$$^{3}$ & 0 & 0 & 4\\
103 & 104 & $(3;8,1)$, $(6;8,8)$$^{3}$ & 0 & 0 & 4\\
107 & 128 & $(1;16,2)$, $(2;2,2)$$^{2}$, $(4;16,16)$$^{4}$ & 0 & 2 & 4\\
109 & 110 & $(3;2,2)$$^{7}$ & 0 & 0 & 4\\
113 & 112 & $(1;4,4)$, $(2;2,2)$$^{2}$, $(4;16,16)$$^{4}$ & 0 & 0 & 4\\
\bottomrule
\end{longtable}
\end{center}

\subsection{Other data}
\begin{itemize}
\item $j(E_3)=12544=2^8\cdot7^2$; $E_3$ = \texttt{392.c1}.
\item $E'$ has $j=-3375$. A minimal model is only determined up to changes of coordinates with unit $u$, and its coefficients depend on the choice of integral basis. So we do not list coefficients here; the accompanying files contain the model used.
\item The complex height regulator of $\langle R_1,R_2,R_3\rangle$ is $\approx5749.646$. That of $\langle S_1,S_2,S_3\rangle$ is $\approx359.353=5749.646/16$, as it should be for a subgroup of index $4$.
\item $\#C(\F_{5581})=6020$. At every prime above $5581$, $E'(\F_{5581})\cong\Z/1432\times\Z/4$.
\item The totally split primes of $F$ below $20000$ are $5581$, $15107$ and $17539$. Of these, $W$ splits completely at every prime above only for $5581$.
\end{itemize}

\section{The computer calculations}\label{app:code}
All computations were done in PARI/GP 2.15.4. The accompanying archive of the second implementation contains the following scripts, in the order in which they are run.
\begin{enumerate}
\item \texttt{t1.gp}, \texttt{t2.gp}: the field $F$, the ramification check of $j_\ns$, \texttt{bnfinit} and \texttt{bnfcertify}.
\item \texttt{smooth.gp}, \texttt{smooth2.gp}: Lemma~\ref{lem:smooth}.
\item \texttt{inv.gp}, \texttt{inv2.gp}, \texttt{inv4.gp}: numerical search for the involutions, exact recognition and the exact check $G\circ\iota=G$, $\iota^2=\mathrm{id}$.
\item \texttt{quot.gp}: the quotient $E'$, the Weierstrass model, the minimal model, the conductor, the torsion, the points $R_i,S_i$ and the height determinant.
\item \texttt{load.gp}, \texttt{sieve.gp}, \texttt{sieve\_main\_body.gp}: the sieve modulo $16$; \texttt{psieve\_body.gp}: the sieve modulo $5581$. (Scripts with suffix \texttt{\_body} are concatenated after \texttt{load.gp} into the \texttt{run\_*.gp} files, which are included.)
\item \texttt{sat*.gp}, \texttt{misc\_main.gp}: saturation certificates, the torsion certificate, the totally split primes, the rank of the matrix modulo $179$.
\item \texttt{ellgen.gp}, \texttt{padic1\_body.gp}, \texttt{padic2\_body.gp}, \texttt{padic3\_body.gp}: sigma functions, local heights, characters, $B_\Sigma$, global heights and the identity at the known points.
\item \texttt{disc.gp}: the power series in the four discs and the Strassmann data.
\item \texttt{search.gp}: the point search.
\end{enumerate}


\end{document}